\documentclass[12pt]{amsart}
\usepackage{amssymb, latexsym, amsmath, amsfonts, hyperref, enumitem}

\newtheorem{thm}{Theorem}[section]

\newtheorem{lem}[thm]{Lemma}
\newtheorem{prop}[thm]{Proposition}
\theoremstyle{definition}

\theoremstyle{remark}
\newtheorem{rem}[thm]{Remark}
\numberwithin{equation}{section}

\newcommand{\mbb}{\mathbb}
\newcommand{\de}{\delta}
\newcommand{\ga}{\gamma}
\newcommand{\ra}{\rightarrow}
\newcommand{\si}{\sigma}

\newcommand{\pa}{\partial}
\newcommand{\ov}{\overline}
\newcommand{\sm}{\setminus}
\newcommand{\ep}{\epsilon}
\newcommand{\no}{\noindent}
\newcommand{\al}{\alpha}
\newcommand{\Om}{\Omega}
\newcommand{\om}{\omega}
\newcommand{\cal}{\mathcal}
\newcommand{\ti}{\tilde}
\newcommand{\la}{\lambda}

\begin{document}
\title{Dynamical properties of families of holomorphic mappings}
\thanks{The first named author was supported by CSIR-UGC(India) fellowship}
\thanks{The second named author was supported by the DST SwarnaJayanti Fellowship 2009--2010 and a UGC--CAS Grant}
\author{Ratna Pal and Kaushal Verma}

\address{Ratna Pal: Department of Mathematics, Indian Institute of Science, Bangalore 560 012, India}
\email{ratna10@math.iisc.ernet.in}

\address{Kaushal Verma: Department of Mathematics, Indian Institute of Science, Bangalore 560 012, India}
\email{kverma@math.iisc.ernet.in}

\pagestyle{headings}

\begin{abstract} We study some dynamical properties of skew products of H\'{e}non maps of $\mbb C^2$ that are fibered over a compact metric space $M$. The problem reduces to
understanding the dynamical behavior of the composition of a pseudo-random sequence of H\'{e}non mappings. In analogy with the dynamics of the iterates of a single H\'{e}non map, it is
possible to construct fibered Green's functions that satisfy suitable invariance properties and the corresponding stable and unstable currents. This analogy is carried forth in two
ways: it is shown that the successive pullbacks of a suitable current by the skew H\'{e}non maps converges to a multiple of the fibered stable current and secondly, this
convergence result is used to obtain a lower bound on the topological entropy of the skew product in some special cases. The other class of maps that are studied are skew products of
holomorphic endomorphisms of $\mbb P^k$ that are again fibered over a compact base. We define the fibered basins of attraction and show that they are pseudoconvex and Kobayashi hyperbolic.

\end{abstract}

\maketitle


\section{Introduction}

\no The purpose of this note is to study various dynamical properties of some classes of fibered mappings. We will first consider families of the form $H : M \times
\mbb C^2 \ra M \times \mbb C^2$ defined by
\begin{equation}
H(\la, x, y) = (\sigma(\la), H_{\la}(x, y))
\end{equation}
where $M$ is an appropriate parameter space, $\sigma$ is a self map of $M$ and for each $\la \in M$, the map
\[
H_{\la}(x, y) = H_{\la}^{(m)} \circ H_{\la}^{(m-1)} \circ \ldots \circ H_{\la}^{(1)}(x, y)
\]
where for every $1 \le j \le m$,
\[
H_{\la}^{(j)}(x, y) = (y, p_{j, \la}(y) - a_{j}(\la) x)
\]
is a generalized H\'{e}non map with $p_{j, \la}(y)$ a monic polynomial of degree $d_j \ge 2$ whose coefficients and $a_{j}(\la)$ are functions on $M$. The degree of $H_{\la}$ is $d =
d_1d_2 \ldots d_m$ which does not vary with $\la$. The two cases that will be considered here are as follows. First, $M$ is a compact metric space
and $\sigma$, $a_j$ and the coefficients of $p_{j, \la}$ are continuous functions on $M$ and second, $M \subset \mbb C^k$, $k \ge 1$ is open in which case $\sigma$,
$a_j$ and the coefficients of $p_{j, \la}$ are assumed to be holomorphic in $\la$. In both cases, $a_j$ is assumed to be a non-vanishing function on $M$. We are interested
in studying the ergodic properties of such a family of mappings. Part of the reason for this choice stems from the
Fornaess-Wu classification (\cite{FW}) of polynomial automorphisms of $\mbb C^3$ of degree at most $2$ according to which any such map is affinely conjugate to

\begin{enumerate}
\item[(a)] an affine automorphism,
\item[(b)] an elementary polynomial automorphism of the form
\[
E(x, y, z) = (P(y, z) + ax, Q(z) + by, cz + d)
\]
where $P, Q$ are polynomials with $\max \{\deg (P), \deg (Q) \} = 2$ and $abc \not= 0$, or
\item[(c)] to one of the following:
\begin{itemize}
\item $H_1(x, y, z) = (P(x, z) + ay, Q(z) + x, cz + d)$
\item $H_2(x, y, z) = (P(y, z) + ax, Q(y) + bz, y)$
\item $H_3(x, y, z) = (P(x, z) + ay, Q(x) + z, x)$
\item $H_4(x, y, z) = (P(x, y) + az, Q(y) + x, y)$
\item $H_5(x, y, z) = (P(x, y) + az, Q(x) + by, x)$
\end{itemize}
where $P, Q$ are polynomials with $\max \{ \deg(P), \deg(Q) \} = 2$ and $abc \not= 0$.
\end{enumerate}

\no The six classes in (b) and (c) put together were studied in \cite{CF} and \cite{CG} where suitable Green functions and associated invariant measures were constructed for them. As
observed in \cite{FW}, several maps in (c) are in fact families of H\'{e}non maps for special values of the parameters $a, b, c$ and for judicious choices of the
polynomials $P, Q$. For instance, if $Q(z) = 0$ and $P(x, z) = x^2 + \ldots$, then $H_1(x, y, z) = (P(x, z) + ay, x, z)$ which is conjugate to
\[
(x, y, z) \mapsto (y, P(y, z) + ax, cz + d) = (y, y^2 + \ldots + ax, cz + d)
\]
by the inversion $\tau_1(x, y, z) = (y, x, z)$. Here $\si(z) = cz + d$. Similarly, if $a = 1, P(y, z) = 0$ and $Q$ is a quadratic polynomial, then $H_2(x, y, z) = (x, Q(y) + bz, y)$
which is conjugate to
\[
(x, y, z) \mapsto (x, z, Q(z) + by) = (x, z, z^2 + \ldots + by)
\]
by the inversion $\tau_3(x, y, z) = (x, z, y)$. Here $\si(x) = x$ and finally, if $b = 1, Q(x) = 0$ and $P(x, y) = x^2 + \ldots$, then $H_5(x, y, z) = (P(x, y) + az, y, x)$ which is
conjugate to
\[
(x, y, z) \mapsto (z, y, P(z, y) + ax) = (z, y, z^2 + \ldots + ax)
\]
by the inversion $\tau_5(x, y, z) = (z, y, x)$ where again $\si(y) = y$. All of these are examples of the kind described in $(1.1)$ with $M = \mbb C$. In the first example, if
$c \not= 1$ then an affine change of coordinates involving only the $z$-variable can make $d = 0$ and if further $\vert c \vert \le 1$, then we may take a closed disc around the origin
in $\mbb C$ which
will be preserved by $\sigma(z) = cz$. This provides an example of a H\'{e}non family that is fibered over a compact base $M$. Further, since the parameter mapping $\sigma$ in the last
two examples is just the identity, we may restrict it to a closed ball to obtain more examples of the case when $M$ is compact.

\medskip

The maps considered in $(1.1)$ are in general $q$-regular, for some $q \ge 1$, in the sense of Guedj--Sibony (\cite{GS}) as the following example shows. Let $\mathcal H : \mbb C^3 \ra
\mbb C^3$ be given by
\[
\mathcal H(\la, x, y) = (\la, y, y^2 - ax), a \not= 0
\]
which in homogeneous coordinates becomes
\[
\mathcal H([\la : x : y : t]) = [\la t : yt : y^2 - axt : t^2].
\]
The indeterminacy set of this map is $I^+ = [\la : x : 0 : 0]$ while that for $\mathcal H^{-1}$ is $I^{-1} = [\la : 0 : y : 0]$. Thus $I^+ \cap I^- = [1 : 0 : 0 : 0]$ and it can be
checked that $X^+ = \ov{\mathcal H \big( (t = 0) \sm I^+ \big)} = [0: 0: 1: 0]$ which is disjoint from $I^+$. Also, $X^- = \ov {\mathcal H^- \big( (t = 0) \sm I^- \big)} =
[0:1:0:0]$ which is disjoint from $I^-$. All these observations imply that $\mathcal H$ is $1$-regular in the sense of \cite{GS}. Further, $\deg(\mathcal H)
= \deg(\mathcal H^{-1}) = 2$. This global view point does have several advantages as the results in \cite{GS}, \cite{G} show. However, thinking of $(1.1)$ as a family of maps was
seconded by the hope that the methods of Bedford--Smillie (\cite{BS1}, \cite{BS2} and
\cite{BS3}) and Fornaess--Sibony \cite{FS} that were developed to handle the case of a single generalized H\'{e}non map would be amenable to this situation -- in fact, they are to a
large extent. Finally, in view of the systematic treatment of families of rational maps of the sphere by Jonsson (see \cite{JM}, \cite{J}), considering families of 
H\'{e}non maps appeared to be a natural next choice. Several pertinent remarks about the family $H$ in (1.1) with $\sigma(\lambda)=\lambda$ can be found in \cite{DS}.

\medskip

Let us first work with the case when $M$ is a compact metric space. For $n \ge 0$, let
\[
H_{\la}^{\pm n} = H_{\si^{n-1}(\la)}^{\pm 1} \circ \cdots \circ H_{\si(\la)}^{\pm 1} \circ H_{\la}^{\pm 1}.
\]
Note that $H_{\la}^{+n}$ is the second coordinate of the $n$-fold iterate of $H(\la, x, y)$. Furthermore
\[
(H_{\la}^{+n})^{-1} = H_{\la}^{-1} \circ H_{\si(\la)}^{-1} \circ \cdots \circ H_{\si^{n-1}(\la)}^{-1} \not= H_{\la}^{-n}
\]
and
\[
(H_{\la}^{-n})^{-1} = H_{\la} \circ H_{\si(\la)} \circ \cdots \circ H_{\si^{n-1}(\la)} \not= H_{\la}^{+n}
\]
for $n \ge 2$. The presence of $\si$ creates an asymmetry which is absent in the case of a single H\'{e}non map and which requires the consideration of these maps as will be seen
later. In what follows, no conditions on $\si$ except continuity are assumed unless stated otherwise.

\medskip

The first thing to do is to construct invariant measures for the family $H(\la, x, y)$ that respect the action of
$\sigma$. The essential step toward this is to construct a uniform filtration $V_R$, $V_R^{\pm}$ for the maps $H_\lambda$ where $R>0$ is sufficiently large.

\medskip

For each $\lambda \in M$, the sets $I_\lambda^{\pm}$ of escaping points and the sets $K_\lambda^{\pm}$ of non-escaping points under random iteration determined by $\sigma$ on $M$ are defined as follows:
\[
I_\lambda^{\pm}=\{z\in \mathbb{C}^2: \Vert H_{\la}^{\pm n}(x, y) \Vert \rightarrow \infty \; \text{as} \;
n\rightarrow \infty \},
\]
\[
K_\lambda^{\pm}=\{z\in \mathbb{C}^2: \; \text{the sequence}\;  \{ H_{\la}^{\pm n} (x, y)\}_n \; \text{is bounded}\}
\]
Clearly, $H_\lambda^{\pm 1}(K_\lambda^{\pm})= K_{\sigma(\lambda)}^{\pm}$ and
$H_\lambda^{\pm 1}(I_\lambda^{\pm})= I_{\sigma(\lambda)}^{\pm}$. Define $K_{\la} = K_{\la}^+ \cap
K_{\la}^-, J_{\la}^{\pm} = \pa K_{\la}^{\pm}$ and $J_{\la} = J_{\la}^+ \cap J_{\la}^-$. For each $\la \in M$ and $n \ge 1$, let
\[
G_{n, \la}^{\pm}(x, y) = \frac{1}{d^n} \log^+ \Vert H_{\la}^{\pm n}(x, y) \Vert
\]
where $\log^+ t=\max \{\log t,0\}$.

\begin{prop}\label{pr1}
The sequence $G_{n, \la}^{\pm}$ converges uniformly on compact subsets of $M \times \mbb C^2$ to a continuous function $G_{\la}^{\pm}$ as $n \ra \infty$ that satisfies
\[
d^{\pm 1} G_{\la}^{\pm} = G_{\si(\la)}^{\pm} \circ H_{\la}^{\pm 1}
\]
on $\mbb C^2$. The functions $G_{\la}^{\pm}$ are positive pluriharmonic on $\mbb C^2 \setminus K_{\la}^{\pm}$, plurisubharmonic on $\mbb C^2$ and vanish precisely on $K_{\la}^{\pm}$.
The correspondence $\lambda \mapsto G_\lambda^{\pm}$ is continuous. In case $\si$ is surjective, $G_{\la}^+$ is locally uniformly H\"{o}lder continuous, i.e., for each compact $S
\subset
\mbb C^2$, there exist constants $\tau, C > 0$ such that
\[
\big\vert G_{\la}^+(x, y) - G_{\la}^+(x', y') \big\vert \le C \Vert (x, y) - (x', y') \Vert^{\tau}
\]
for all $(x, y), (x', y') \in S$. The constants $\tau, C$ depend on $S$ and the map $H$ only.
\end{prop}

\no As a result, $\mu_{\la}^{\pm} = (1/2\pi) dd^c G_{\la}^{\pm}$ are well defined positive closed $(1, 1)$ currents on $\mbb C^2$ and hence $\mu_{\la} = \mu_{\la}^+ \wedge \mu_{\la}^-$
defines a probability measure on $\mbb C^2$ whose support is contained in $V_R$ for every $\la \in M$. Moreover the correspondence $\lambda\mapsto \mu_\lambda$ is continuous. That these
objects are well behaved under the pullback and
push forward operations by $H_{\la}$ and at the same time respect the action of $\si$ is recorded in the following:

\begin{prop}\label{pr2}
With $\mu_{\la}^{\pm}, \mu_{\la}$ as above, we have
\[
{(H_{\la}^{\pm 1})}^{\ast} \mu_{\si(\la)}^{\pm} = d^{\pm 1} \mu_{\la}^{\pm}, \; (H_{\la})_{\ast} \mu_{\la}^{\pm} = d^{\mp 1} \mu_{\si(\la)}^{\pm}
\]
The support of $\mu^{\pm}_{\la}$ equals $J_{\la}^{\pm}$ and the correspondence $\la \mapsto J_{\la}^{\pm}$ is lower semi-continuous. Furthermore, for each $\lambda\in
M$, the pluricomplex Green function of $K_\lambda$ is $\max\{G_\lambda^+, G_\lambda^-\}$, $\mu_\lambda$ is the complex
equilibrium measure of $K_\lambda$ and ${\rm supp}(\mu_\lambda)\subseteq J_\lambda$.

\medskip

In particular, if $\si$ is the identity on $M$, then $(H_{\la}^{\pm 1})^{\ast} \mu_{\la} = \mu_{\la}$.
\end{prop}

\no Let $T$ be a positive closed $(1, 1)$ current in a domain $\Om \subset \mbb C^2$ and let $\psi \in C^{\infty}_0(\Om)$ with $\psi \ge 0$ be such that $\text{supp}(\psi) \cap
\text{supp}(dT) = \phi$. Theorem 1.6 in \cite{BS3} shows that for a single H\'{e}non map $H$ of degree $d$, the sequence $d^{-n} H^{n \ast}(\psi T)$ always converges to $c \mu^+$
where $c = \int \psi T \wedge \mu^- > 0$. In the same vein, for each $\la \in M$ let $S_{\la}(\psi, T)$ be the set of all possible limit points of the sequence
$d^{-n}\big( H_{\la}^{+n}\big)^{\ast}(\psi T)$.

\begin{thm}\label{thm1}
$S_{\la}(\psi, T)$ is nonempty for each $\la \in M$ and $T, \psi$ as above. Each $\ga_{\la} \in S_{\la}(\psi, T)$ is a positive multiple of $\mu_{\la}^+$.
\end{thm}

In general, $S_\lambda(\psi, T)$ may be a large set. However, there are two cases for which it is possible to
determine the cardinality of $S_{\la}(\psi, T)$ and both are illustrated by the examples mentioned earlier.

\begin{prop}
If $\sigma$ is the identity on $M$ or when $\sigma : M \ra M$ is a contraction, i.e., there exists $\la_0 \in M$ such that $\si^n(\la) \ra \la_0$ for all $\la \in M$, the set
$S_{\la}(\psi, T)$ consists of precisely one element. Consequently, in each of these cases there exists a constant $c_{\la}(\psi, T) > 0$ such that
\[
\lim_{n \ra  \infty} d^{-n} \big( H_{\la}^{+n}\big)^{\ast}(\psi T) = c_{\la}(\psi, T) \mu_{\la}^+.
\]
\end{prop}

\medskip

Let us now consider the case when $M$ is a relatively compact open subset of $\mbb C^k$, $k \ge 1$ and the map $\si$ is the identity on $M$. Since this means that the slices
over each point in $M$ are preserved, we may (by shrinking $M$ slightly) assume that the maps $H_{\la}$ are well defined in a neighborhood of $\ov M$. Thus the earlier discussion
about the construction of $\mu_{\la}^{\pm}, \mu_{\la}$ applies to the family (which will henceforth be considered)
\begin{gather}
H : M \times \mbb C^2 \ra M \times  \mbb C^2, \notag\\
H(\la, x, y) = (\la, H_{\la}(x, y)) \notag.
\end{gather}

\no For every probability measure $\mu'$ on $M$,
\begin{equation}
\langle \mu, \phi \rangle = \int_{M} \bigg( \int_{\{\la\} \times \mbb C^2} \phi \; \mu_{\la} \bigg) \mu'(\la)
\end{equation}
defines a measure on $M \times \mbb C^2$ by describing its action on continuous functions $\phi$ on $M \times \mbb C^2$. This is not a dynamically natural measure 
since $\mu'$ is arbitrary. It will turn out that the support of $\mu$ is contained in
\[
\cal J = \bigcup_{\la \in M} \left( \{ \la \} \times J_{\la} \right) \subset M \times V_R.
\]
The slice measures of $\mu$ are evidently $\mu_{\la}$ and
since $\si$ is the identity it can be seen from Proposition 1.2 that $\mu$ is an invariant probability measure for $H$ as above.

\begin{thm}
Regard $H$ as a self map of ${\rm supp}(\mu$) with invariant measure $\mu$. The measure theoretic entropy of $H$ with respect to $\mu$ is at least $\log d$. In particular, the topological
entropy of $H : \cal J \ra \cal J$ is at least $\log d$.
\end{thm}

It would be both interesting and desirable to obtain lower bounds for the topological entropy for an arbitrary continuous function $\sigma$ in (1.1).

\medskip

We will now consider continuous families of holomorphic endomorphisms of $\mbb P^k$. For a compact metric space $M$, $\si$ a continuous self map of $M$, define $F : M \times \mbb P^k
\ra M \times \mbb P^k$ as
\begin{equation}
F(\la, z) = (\si(\la), f_{\la}(z))
\end{equation}
where $f_{\la}$ is a holomorphic endomorphism of $\mbb P^k$ that depends continuously on $\la$. Each $f_{\la}$ is assumed to have a fixed degree $d \ge 2$. Corresponding to each
$f_{\la}$  there exists a non-degenerate homogeneous holomorphic mapping $F_{\la} : \mbb C^{k+1} \ra \mbb C^{k+1}$ such that $\pi \circ F_{\la} = f_{\la} \circ \pi$ where
$\pi : \mbb C^{k+1} \setminus \{0\} \ra \mbb P^k$ is the canonical projection. Here, non-degeneracy means that $F_{\la}^{-1}(0) = 0$ which in turn implies that
there are uniform constants $l, L >0$ with
\begin{eqnarray}
l \Vert x \Vert^d \le \Vert F_{\la}(x) \Vert \le L \Vert x \Vert^d
\end{eqnarray}
for all $\la \in M$ and $x \in \mbb C^{k+1}$. Therefore for $0 < r \leq (2L)^{-1/(d-1)}$
\[
\Vert F_\lambda(x) \Vert \leq (1/2) \Vert x \Vert
\]
for all $\lambda\in M$ and $\Vert x \Vert \leq r$. Likewise for $R\geq (2l)^{-1/(d-1)}$
\[
\Vert F_\lambda(x) \Vert \geq 2 \Vert x \Vert
\]
for all $\lambda\in M$ and  $\Vert x \Vert \geq R$.

\medskip

While the ergodic properties of such a family have been considered in \cite{T1}, \cite{T2} for instance, we are interested in looking at
the basins of attraction which may be defined for each $\la \in M$ as
\[
\mathcal A_{\la} =  \big\{ x \in \mbb C^{k+1} : F_{\si^{n-1}(\la)} \circ \ldots \circ F_{\si(\la)} \circ F_{\la}(x) \ra 0 \; \text{as} \; n \ra \infty \big\}
\]
and for each $\lambda\in M$ the region of normality $\Om'_{\la} \subset \mbb P^k$ which consists of all points $z \in \mbb P^k$ for which there is a neighborhood $V_z$ on which the sequence
$\big \{f_{\si^{n-1}(\la)} \circ \ldots \circ f_{\si(\la)} \circ f_{\la} \big\}_{n \ge 1}$ is normal. Analogs of $\mathcal A_{\la}$ arising from composing a given sequence of
automorphisms of $\mbb C^n$ have been considered in \cite{PW} where an example can be found for which these are not open in $\mbb C^n$. However, since each $F_{\la}$ is homogeneous, it
is straightforward to verify that each $\cal A_{\la}$ is a nonempty, pseudoconvex complete circular domain. As in the
case of a single holomorphic endomorphism of $\mbb P^k$ (see \cite{HP}, \cite{U}), the link between these two domains is provided by the Green function.

\medskip

For each $\la \in M$ and $n \ge 1$, let
\[
G_{n, \la}(x) = \frac{1}{d^n} \log \Vert F_{\si^{n-1}(\la)} \circ \ldots \circ F_{\si(\la)} \circ F_{\la}(x) \Vert.
\]

\begin{prop}
For each $\la \in M$, the sequence $G_{n, \la}$ converges uniformly on $\mbb C^{k+1}$ to a continuous plurisubharmonic function $G_{\la}$ which satisfies
\[
G_{\la}(c x) = \log \vert c \vert + G_{\la}(x)
\]
for $c \in \mbb C^{\ast}$. Further, $d G_{\la} = G_{\si(\la)} \circ F_{\la}$, and $G_{\la_n} \ra G_{\la}$ locally uniformly on $\mbb C^{k+1} \setminus \{0\}$ as $\la_n \ra \la$ in
$M$. Finally,
\[
\cal A_{\la} = \{x \in \mbb C^{k+1} : G_{\la}(x) < 0\}
\]
for each $\la \in M$.
\end{prop}



For each $\la \in M$, let $\mathcal H_{\la} \subset \mbb C^{k+1}$ be the collection of those points in a neighborhood of which $G_{\la}$ is pluriharmonic and define $\Om_{\la} =
\pi(\mathcal H_{\la}) \subset \mbb P^k$.

\begin{prop}
For each $\la \in M$, $\Om_{\la} = \Om'_{\la}$. Further, each $\Om_{\la}$ is pseudoconvex and Kobayashi hyperbolic.
\end{prop}

{\bf{Acknowledgment:}} The first named author would like to thank G. Buzzard and M. Jonsson for their helpful comments in an earlier version of this paper.




\section{Fibered families of H\'{e}non maps}

\no The existence of a filtration $V^{\pm}_R, V_R$ for a H\'{e}non map is useful in localizing its dynamical behavior. To study a family of such maps, it is therefore essential to first establish the existence of a uniform filtration that works for all of them. Let
\begin{align*}
V_R^+ &= \big\{ (x, y) \in \mbb C^2 : \vert y \vert > \vert x \vert, \vert y \vert > R \big\},\\
V_R^- &= \big\{ (x, y) \in \mbb C^2 : \vert y \vert < \vert x \vert, \vert x \vert > R \big\},\\
V_R   &= \big\{ (x, y) \in \mbb C^2 : \vert x \vert, \vert y \vert \le R\}
\end{align*}
be a filtration of $\mathbb{C}^2$ where $R$ is large enough so that
\[
H_{\la}(V_R^+) \subset V_R^+
\]
for each $\la \in M$. The existence of such an $R$ is shown in the following lemma.
\begin{lem} \label{le1}
There exists $R>0$ such that
$$
H_\lambda(V_R^+)\subset V_R^+, \ \ H_\lambda(V_R^+\cup V_R)\subset V_R^+\cup V_R
$$
and
$$
H_\lambda^{-1}(V_R^-)\subset V_R^-, \ \ H_\lambda^{-1}(V_R^-\cup V_R)\subset V_R^-\cup V_R
$$
for all $\lambda \in M$. Furthermore,
$$
I_\lambda^{\pm}=\mathbb{C}^2\setminus K_\lambda^{\pm}=\bigcup_{n=0}^\infty (H_{\la}^{\pm n})^{-1}(V_R^{\pm}).
$$
\end{lem}
\begin{proof}
Let
\[
p_{j,\lambda}(y)=y^{d_j} + c_{\lambda(d_j-1)}y^{d_j-1} + \ldots + c_{\lambda 1}y + c_{\lambda 0}
\]
be the polynomial that occurs in the definition of $H_\lambda^{(j)}$. Then

\begin{equation}
\vert y^{-d_j} p_{j, \la}(y) - 1 \vert \le \vert c_{\la(d_j - 1)} y^{-1} \vert + \ldots + \vert c_{\la 1} y^{-d_j + 1} \vert + \vert c_{\la 0} y^{-d_j} \vert. \label{0}
\end{equation}

Let $a=\sup_{\lambda,j}|a_{\lambda,j}|$. Since the coefficients of $p_{j,\lambda}$ are continuous on $M$, which is assumed to be compact, and $d_j \ge 2$ it follows that there exists
$R>0$ such that
\[
\vert p_{j,\lambda}(y)  \vert \geq (2 + a) \vert y \vert
\]
for $\vert y \vert>R$, $\lambda\in M$ and $1\leq j \leq m$. To see that $H_\lambda(V_R^+)\subset V_R^+$ for this $R$, pick $(x,y)\in V_R^+$. Then
\begin{equation}
\lvert p_{j,\lambda}(y)-a_j(\lambda)x\rvert \geq \lvert p_{j,\lambda}(y)\rvert -\lvert a_j(\lambda)x\rvert
\geq \lvert y \rvert \label{1}
\end{equation}
for all $1\leq j \leq m$. It follows that the second coordinate of each $H_\lambda^{(j)}$ dominates the first one. This implies that
\[
H_\lambda(V_R^+)\subset V_R^+
\]
for all $\lambda\in M$. The other invariance properties follow by using similar arguments.

\medskip

Let $\rho>1$ be such that
\[
\lvert p_{j,\lambda}(y)-a_j(\lambda)x \rvert > \rho \lvert y \rvert
\]
for $(x,y)\in \overline{V_R^+}$, $\lambda\in M$ and $1\leq j \leq m$. That such a $\rho$ exists follows from  (\ref{1}). By letting $\pi_1$ and $\pi_2$ be the projections on the first
and second coordinate respectively, one can conclude inductively that
\begin{equation}
H_\lambda(x,y)\in V_R^+ \text{ and } \vert \pi_2(H_\lambda(x,y)) \vert >\rho^m \vert y \vert. \label{2}
\end{equation}
Analogously, for all $(x,y)\in \overline{V_R^{-}}$ and for all $\lambda\in M$, there exists a  $\rho>1$ satisfying
\begin{equation}
H_\lambda^{-1}(x,y)\in V_R^- \text{ and }|\pi_1(H_\lambda^{-1}(x,y))|>\rho^m|x|.\label{2.1}
\end{equation}
These two facts imply that
\begin{equation}
\overline{V_R^+} \subset H_{\la}^{-1}(\overline{V_R^+})\subset H_{\la}^{-1} \circ H_{\si(\la)}^{-1} (\ov{V_R^+}) \subset  \ldots \subset (H_{\la}^{+n})^{-1}(\overline{V_R^+})\subset
\ldots
\end{equation}
and
\begin{equation}
\overline{V_R^{-}} \supset H_{\la}^{-1}(\overline{V_R^{-}})\supset H_{\la}^{-1} \circ H_{\si(\la)}^{-1}(\ov{V_R^-}) \supset \ldots \supset (H_{\la}^{+n})^{-1}(\overline{V_R^-})\supset
\ldots .
\label{3}
\end{equation}

\medskip

At this point one can observe that if we start with a point in $\overline{V_R^+}$, it eventually escapes toward the point at infinity under forward iteration determined by the
continuous function $\sigma$, i.e., $\vert H_{\la}^{+n}(x, y) \vert \rightarrow \infty$ as $n\rightarrow \infty$. This can be
justified by using (\ref{2}) and observing that
\begin{equation*}
\lvert y_\lambda^n \rvert > \rho^m \lvert y_\lambda^{n-1} \rvert> \rho^{2m}\lvert y_\lambda^{n-2} \rvert> \ldots >\rho^{nm}\lvert y \rvert>\rho^{nm}R
\end{equation*}
where $H_{\la}^{+n}(x, y) =(x_\lambda^n,y_\lambda^n)$. A similar argument shows that if we start with any point in $(x,y)\in
\bigcup_{n=0}^{\infty} (H_{\la}^{+n})^{-1}(V_R^+)$ the orbit of the point never remains bounded. Therefore
\begin{equation}
\bigcup_{n=0}^{\infty} (H_{\la}^{+n})^{-1}(V_R^+)\subseteq I_\lambda^+.
\end{equation}
 Moreover using (\ref{2}) and (\ref{2.1}), we get
\[
(H_{\la}^{-n})^{-1}(V_R^+)\subseteq \big\{(x,y):\lvert y\rvert > \rho^{nm}R \big\}
\]
and
\[
(H_{\la}^{+n})^{-1}(V_R^-)\subseteq \big\{(x,y):\lvert x \rvert > \rho^{nm}R \big\}
\]
which give
\begin{equation}
\bigcap_{n=0}^{\infty} (H_{\la}^{-n})^{-1}(V_R^+) = \bigcap_{n=0}^{\infty} (H_{\la}^{-n})^{-1}(\overline{V_R^+})=\phi \label{4}
\end{equation}
and
\begin{equation}
\bigcap_{n=0}^{\infty} (H_{\la}^{+n})^{-1}(V_R^-)= \bigcap_{n=0}^{\infty} (H_{\la}^{+n})^{-1}(\overline{V_R^{-}})=\phi \label{4.1}
\end{equation}
respectively. Set
\[
W_R^+=\mathbb{C}^2\setminus \overline{V_R^{-}} \text{ and }W_R^-=\mathbb{C}^2\setminus \overline{V_R^+}.
\]
Note that (\ref{3}) and (\ref{4.1}) are equivalent to
\begin{equation}
W_R^+\subset H_{\lambda}^{-1}(W_R^+) \subset \ldots \subset (H_{\la}^{+n})^{-1}(W_R^+)\subset \ldots
\end{equation}
and
\begin{equation}
\bigcup_{n=0}^{\infty} (H_{\la}^{+n})^{-1}(W_R^+)= \mathbb{C}^2 \label{5}
\end{equation}
respectively. Now (\ref{5}) implies that for any point $(x,y)\in \mathbb{C}^2$ there exists $n_0>0$ such that $H_{\la}^{+n}(x,y)\in
W_R^+\subset V_R\cup \overline{V_R^+}$ for all $n\geq n_0$. So either
\[
H_{\la}^{+n}(x,y)\in V_R
\]
for all $n \ge n_0$ or there exists $n_1 \geq n_0$ such that $H_{\la}^{+n_1}(x,y)\in \overline{V_R^+}$. In the latter case,
$H_{\la}^{+(n_1+1)}(x,y)\in V_R^+$ by (\ref{2}). This implies that
\begin{equation*}
I_\lambda^{+}=\mathbb{C}^2\setminus K_\lambda^{+}=\bigcup_{n=0}^\infty (H_{\la}^{+n})^{-1}(V_R^{+}).\label{5.1}
\end{equation*}
A set of similar arguments yield
\begin{equation*}
I_\lambda^{-}=\mathbb{C}^2\setminus K_\lambda^{-}=\bigcup_{n=0}^\infty (H_{\la}^{-n})^{-1}(V_R^{-}).
\end{equation*}
\end{proof}

\begin{rem}\label{re1}
It follows from Lemma \ref{le1} that for any compact $A_\lambda \subset \mathbb{C}^2$ satisfying $A_\la \cap K_\lambda^+=\phi$, there exists $N_\lambda>0$
such that $H_{\la}^{+n_{\la}}(A_\lambda)\subseteq V_R^+$. More generally, for any compact $A \subset \mathbb{C}^2$
that satisfies $A\cap K_\lambda^+=\phi$ for each $\lambda\in M$,
there exists $N>0$ so that $H_{\la}^{+N}(A)\subseteq V_R^+$ for all $\lambda\in M$. The proof again relies on the fact that
the coefficients of $p_{j,\lambda}$ and $a_j(\lambda)$
vary continuously in $\lambda$ on the compact set $M$ for all $1 \le j \le m$.
\end{rem}

\begin{rem}\label{re2}
By applying the same kind of techniques as in the case of a single H\'{e}non map, it is possible to show that $I_\lambda^{\pm}$ are nonempty, pseudoconvex domains and $K_\lambda^{\pm}$
are closed sets satisfying $K_\lambda^{\pm}\cap V_R^{\pm}=\phi$ and having nonempty intersection with the $y$-axis and $x$-axis respectively. In particular, $K_\lambda^{\pm}$ are
nonempty and unbounded.
\end{rem}

\subsection*{Proof of Proposition \ref{pr1}}
Since the polynomials $p_{j, \la}$ are all monic, it follows that for every small $\ep_1 > 0$ there is a large enough $R > 1$ so that for all
$(x,y)\in \overline{V_R^+}$, $1\leq j \leq m$ and for all $\lambda\in M$, we have $H_\lambda^{(j)}(x,y)\in V_R^+$
and
\begin{equation}
(1-\ep_1)\lvert y \rvert^{d_j}<\lvert \pi_2\circ H_\lambda^{(j)}(x,y)\rvert < (1+\ep_1)\lvert y \rvert^{d_j}. \label{7}
\end{equation}
For a given $\ep > 0$, choose an $\epsilon_1>0$ small enough so that the constants
\[
A_1=\prod_{j=1}^m (1-\epsilon_1)^{d_{j+1} \ldots d_m} \text{ and }
A_2=\prod_{j=1}^m (1+\epsilon_1)^{d_{j+1} \ldots d_m}
\]
(where $d_{j+1} \ldots d_m=1$ by definition when $j=m$) satisfy $1-\epsilon \leq A_1$ and $A_2 \leq 1+\epsilon$. Therefore by applying (\ref{7}) inductively, we get
\begin{equation}
(1-\epsilon)\lvert y \rvert^{d} \leq A_1\lvert y \rvert^{d}<\lvert \pi_2\circ H_\lambda(x,y)\rvert<A_2\lvert y \rvert^{d}\leq (1+\epsilon)\lvert y \rvert^{d}\label{8}
\end{equation}
for all $\lambda\in M$ and for all $(x,y)\in \overline{V_R^+}$. Let $(x,y)\in \overline{V_R^+}$. In view of (\ref{2}) there exists a large $R>1$ so that
$H_\lambda^{+n}(x,y)=(x_\lambda^n,y_\lambda^n)\in V_R^+$ for all $n\geq 1$ and for all $\lambda\in M$. Therefore
$$
G_{n,\lambda}^+(x,y)=\frac{1}{d^n}\log\lvert \pi_2\circ H_\lambda^{+n}(x,y)\rvert
$$
and by applying (\ref{8}) inductively we obtain
\begin{equation*}
(1-\epsilon)^{1+d+ \ldots +d^{n-1}} \lvert y \rvert^{d^n}<\lvert y_\lambda^n \rvert<
(1+\epsilon)^{1+d+ \ldots +d^{n-1}}\lvert y \rvert^{d^n}.
\end{equation*}
Hence
\begin{equation}
0<\log\lvert y \rvert+K_1<G_{n,\lambda}^+(x,y)=\frac{1}{d^n}\log\lvert \pi_2\circ H_\lambda^{+n}(x,y)\rvert<\log\lvert y \rvert+K_2,\label{9}
\end{equation}
with $K_1= (d^n-1)/(d^n(d-1)) \log(1-\epsilon)$ and $K_2= (d^n-1)/(d^n(d-1)) \log(1+\epsilon)$.

\medskip

By (\ref{9}) it follows that
$$
\lvert G_{n+1,\lambda}^+(x,y)-G_{n,\lambda}^+(x,y)\rvert=\left | d^{-n -1} \log\lvert {y_\lambda^{n+1}}/{(y_\lambda^n)^d}\rvert\right | \lesssim d^{-n-1}
$$
which proves that $\{G_{n,\lambda}^+\}$ converges uniformly on $\overline{V_R^+}$. As a limit of a sequence of uniformly convergent pluriharmonic functions $\{G_{n,\lambda}^+\}$,
$G_\lambda^+$ is also pluriharmonic for each $\lambda\in M$ on $V_R^+$. Again by (\ref{9}), for each $\lambda\in M$,
\[
G_\lambda^+-\log\lvert y \rvert
\]
is a bounded pluriharmonic function in $\overline{V_R^+}$. Therefore its restriction to vertical lines of the form $x = c$ can be continued across the point $(c, \infty)$
as a pluriharmonic function. Since
\[
\lim_{\lvert y \rvert\rightarrow \infty}(G_\lambda^+(x,y)-\log\lvert y \rvert)
\]
is bounded in $x\in \mathbb{C}$ by (\ref{9}) it follows that $\lim_{\lvert y \rvert\rightarrow \infty}(G_\lambda^+(x,y)-\log\lvert y \rvert)$ must be a constant, say $\gamma_\lambda$ which also satisfies
$$
\log (1-\epsilon)/(d-1) \leq \gamma_\lambda \leq
\log (1+\epsilon)/(d-1).
$$
As $\epsilon > 0$ is arbitrary, it follows that
\begin{equation}
G_\lambda^+(x,y)=\log\lvert y \rvert + u_\lambda(x,y) \label{9.1}
\end{equation}
on $V_R^+$ where $u_\lambda$ is a bounded pluriharmonic function satisfying $u_\lambda(x,y) \ra 0$ as $\vert y \vert \ra \infty$.

\medskip

Now fix $\lambda\in M$ and $n\geq 1$. For any $r > n$
\begin{eqnarray*}
G_{r,\lambda}^+(x,y)&=& d^{-r}{\log}^+\lvert H_\lambda^{+r}(x,y)\rvert \\
&=& d^{-n}G_{(r-n),\sigma^n(\lambda)}^+\circ H_\lambda^{+n}(x,y).
\end{eqnarray*}

As $r\rightarrow \infty$, $G_{r,\lambda}^+$ converges uniformly on $(H_{\la}^{+n})^{-1}(V_R^+)$ to the pluriharmonic function  $d^{-n} G_{\sigma^n(\lambda)}^+ \circ H_\lambda^{+n}$.
Hence
$$
d^n G_\lambda^+(x,y)=G_{\sigma^n(\lambda)}^+\circ H_\lambda^{+n}(x,y)
$$
for $(x,y)\in (H_\lambda^{+n})^{-1}(V_R^+)$. By (\ref{9}), for $(x,y)\in (H_\lambda^{+n})^{-1}(V_R^+)$
$$
G_{r,\lambda}^+(x,y)=d^{-n}G_{(r-n),\sigma^n(\lambda)}^+\circ H_\lambda^{+n}(x,y)> d^{-n}(\log R + K_1)> 0,
$$
for each $r>n$ which shows that
\[
G_\lambda^+(x,y)\geq d^{-n}(\log R + K_1)>0
\]
for $(x,y)\in (H_\lambda^{+n})^{-1}(V_R^+)$. This is true for each $n\geq 1$. Hence $G_{r,\lambda}^+$ converges uniformly to the pluriharmonic function $G_\lambda^+$ on every compact
set
of
\[
\bigcup_{n=0}^\infty (H_\lambda^{+n})^{-1}(V_R^+)=\mathbb{C}^2\setminus K_\lambda^+.
\]
Moreover $G_\lambda^+ >0$ on $\mathbb{C}^2\setminus K_\lambda^+$.

\medskip

Note that for each $\lambda\in M$, $G_\lambda^+=0$ on $K_\lambda^+$.
By Remark \ref{re2}, there exists a large enough $R>1$ so that $K_\lambda^+\subseteq V_R \cup V_R^-$ for all $\lambda\in M$. Now choose any $A>R>1$. We will show that
$\{G_{n,\lambda}^+\}$ converges uniformly to $G_\lambda^+$ on the bidisc
\[
\Delta_A=\{(x,y):\lvert x \rvert\leq A,\lvert y \rvert\leq A\}
\]
as $n\ra\infty$. Consider the sets
\[
N=\{(x,y)\in \mathbb{C}^2: \lvert x \rvert \leq A\}, \;N_\lambda=N\cap K_\lambda^+
\]
for each $\lambda\in M$. Start with any point $z=(x_0,x_1)\in \mathbb{C}^2$ and define
$(x_i^\lambda,x_{i+1}^\lambda)$ for $\lambda\in M$ and $i\geq 1$ in the following way:
\[
(x_0^\lambda,x_1^\lambda) \xrightarrow{H_\lambda^{(1)}} (x_1^\lambda,x_2^\lambda) \xrightarrow{H_\lambda^{(2)}} \ldots \xrightarrow{H_\lambda^{(m)}}
(x_m^\lambda,x_{m+1}^\lambda)\xrightarrow{H_\lambda^{(1)}} (x_{m+1}^\lambda,x_{m+2}^\lambda)\ra \ldots ,
\]
where $(x_0^\lambda, x_1^\lambda)=(x_0,x_1)$ and we apply $H_\lambda^{(1)}, \ldots ,H_\lambda^{(m)}$ periodically for all $\lambda\in M$. Inductively one can show that if
$(x_i^\lambda,x_{i+1}^\lambda)\in N_\lambda$ for $0\leq i \leq j-1$, then $\lvert x_i^\lambda\rvert \leq A $ for $0\leq i \leq j$.


\medskip

This implies that there exists $n_0>0$ independent of $\lambda$ so that
\begin{equation}
G_{n,\lambda}^+(x,y)< \epsilon \label{10}
\end{equation}
for all $n\geq n_0$ and for all $(x,y)\in N_\lambda$.
Consider a line segment
\[
L_a=\{(a,w):\lvert w \rvert\leq A\} \subset \mbb C^2
\]
with $\lvert a \rvert \leq A$. Then $G_{n,\lambda}^+-G_\lambda^+$ is harmonic on
$L_a^\lambda=\{(a,w):\lvert w \rvert < A\}\setminus K_\lambda^+$ viewed as a subset of $\mathbb{C}$ and the boundary of $L_a^\lambda$ lies in $\{(a,w):\lvert w \rvert=A\}\cup
(K_\lambda^+\cap L_a)$. By Remark \ref{re1}, there exists $n_1>0$ so that
$$
-\epsilon< G_{n,\lambda}^+(a,w)-G_{\lambda}^+(a,w)<\epsilon
$$
for all $n\geq n_1$ and for all $(a,w)\in\{\lvert a \rvert\leq A,\lvert w \rvert=A\}$. The maximum principle shows that
$$
-\epsilon <G_{n,\lambda}^+(x,y)-G_\lambda^+(x,y) < \epsilon
$$
for all  $n\geq \max\{n_0,n_1\}$ and for all  $(x,y)\in L_a^\lambda$. This shows that for any given $\epsilon>0$ there exists $n_2>0$ such that
\begin{equation}
-\epsilon< G_{n,\lambda}^+(z)-G_\lambda^+(z)<\epsilon \label{10.5}
\end{equation}
for all $n\geq n_2$ and for all $(\lambda,z)\in M\times \Delta_A$.

\medskip

Hence $G_{n,\lambda}^+$ converges uniformly to $G_\lambda^+$ on any compact subset of $\mathbb{C}^2$ and this convergence is also uniform with respect to $\lambda\in M$.
In particular this implies that for each $\lambda\in M$, $G_\lambda^+$ is continuous  on $\mathbb{C}^2$ and pluriharmonic on $\mathbb{C}^2\setminus K_\lambda^+$. Moreover $G_\lambda^+$
vanishes on $K_\lambda^+$. In particular, for each $\lambda\in M$, $G_\lambda^+$ satisfies the submean value property on $\mathbb{C}^2$. Hence $G_\lambda^+$ is
plurisubharmonic on $\mathbb{C}^2$.

\medskip

Next, to show that the correspondence $\lambda \mapsto G_\lambda^{\pm}$ is continuous, take a compact set $S\subset \mathbb{C}^2$ and $\lambda_0\in M$. Then
\begin{multline*}
\vert G_{\la}^+(x,y)-G_{\la_0}^+(x,y) \vert \le \vert G_{n,\la}^+(x,y)-G_{\la}^+(x,y)\vert + \vert G_{n,\lambda}^+(x,y)-G_{n,\lambda_0}^+(x,y)\vert \\
                                              + \vert G_{n,\lambda_0}^+(x,y)-G_{\lambda_0}^+(x,y)\vert
\end{multline*}
for $(x,y)\in S$. It follows from (\ref{10.5}) that for given $\epsilon>0$, one can choose a large $n_0>0$ such that the first and third  terms above are less that $\ep/3$. By
choosing $\lambda$ close enough to $\lambda_0$ it follows that $G_{n_0,\lambda}^+(x,y)$ and $G_{n_0,\lambda_0}^+(x,y)$ do not differ by more than $\ep/3$. Hence the correspondence
$\lambda\mapsto G_\lambda^{+}$ is continuous. Similarly, the correspondence $\lambda\mapsto G_\lambda^-$ is also continuous.

\medskip

To prove that $G_\lambda^+$ is H\"{o}lder continuous for each $\lambda\in M$, fix a compact $S \subset \mbb C^2$ and let $R > 1$ be such that $S$ is compactly contained in $V_R$. Using
the continuity of $G_{\la}^+$ in $\la$, there exists a $\de > 0$ such that $G_{\la}^+(x, y) > (d + 1)\de$ for each $\la \in M$ and $(x, y) \in V_R^+$.
Now note that the correspondence $\la \mapsto K_{\la}^+ \cap V_R$ is upper semi-continuous. Indeed, if this is not the case, then there exists a $\la_0 \in M$, an $\ep > 0$ and a sequence
$\la_n \in M$ converging to $\la_0$ such that for each $n \ge 1$ there exists a point $a_n \in K_{\la_n}^+ \cap V_R$ satisfying $\vert a_n - z \vert \ge \ep$ for all $z \in
K_{\la_0}^+$. Let $a$ be a limit point of the $a_n$'s. Then by the continuity of $\la \mapsto G_{\la}^+$ it follows that
\[
0 = G_{\la_n}^+(a_n) \ra G_{\la}^+(a)
\]
which implies that $a \in K_{\la_0}^+$. This is a contradiction. For each $\lambda\in M$, define
\[
\Omega_\delta^\lambda= \big\{ (x,y)\in V_R : \delta < G_\lambda^+(x,y) \leq d \delta \big\}
\]
and
\[
C_\lambda=\sup\big\{ \lvert {\partial G_\lambda^+}/{\partial x}\rvert,\lvert {\partial G_\lambda^+}/{\partial y}\rvert :(x,y)\in \Omega_\delta^\lambda \big\}.
\]
The first observation is that the $C_{\la}$'s are uniformly bounded above as $\la$ varies in $M$. To see this, fix $\la_0 \in M$ and $\tau > 0$ and let $W \subset M$ be a neighbourhood
of $\la_0$ such that the sets
\[
\Om_W = \ov{\bigcup_{\la \in W} \Om_{\de}^{\la}} \;\; \text{and} \;\; K_W = \ov{ \bigcup_{\la \in W} (K_{\la}^+ \cap V_R)}
\]
are separated by a distance of at least $\tau$. This is possible since $K_{\la}^+ \cap V_R$ is upper semicontinuous in $\la$. For each $\la \in W$, $G_{\la}^+$ is pluriharmonic on a
fixed slightly larger open set containing $\Om_W$. Cover the closure of this slightly larger open set by finitely many open balls and on each ball, the mean value property shows that
the derivatives of $G_{\la}^+$ are dominated by a universal constant times the sup norm of $G_{\la}^+$ on it -- and this in turn is dominated by the number of open balls (which is the
same for all $\la \in W$) times the sup norm of $G_{\la}^+$ on $V_R$ upto a univeral constant. Since $G_{\la}^+$ varies continuously in $\la$, it follows that the $C_{\la}$'s are
uniformly bounded for $\la \in W$ and the compactness of $M$ gives a global bound, say $C > 0$ independent of $\la$.

\medskip

Fix $\la_0 \in M$ and pick $(x, y) \in S \setminus K_{\la_0}^+$. Let $N > 0$ be such that
\[
d^{-N} \de < G_{\la_0}^+(x, y) \le d^{-N + 1} \de
\]
so that
\[
\de < d^N G_{\la_0}^+(x, y) \le d \de.
\]
The assumption that $N > 0$ means that $(x, y)$ is very close to $K_{\la_0}^+$. But
\[
d^N G_{\la_0}^+(x, y) = G_{\si^N(\la_0)}^+ \circ H_{\la_0}^{+N}(x, y)
\]
which implies that $H_{\la_0}^{+N}(x, y) \in \Om_{\de}^{\si^N(\la_0)}$ where $G_{\si^N(\la_0)}^+$ is pluriharmonic. Note that
\[
H_{\la_0}(V_R \cup V_R^+) \subset V_R \cup V_R^+, \; H_{\la_0}(V_R^+) \subset V_R^+
\]
which shows that $H_{\la_0}^{+k} \in V_R$ for all $k \le N$ since all the $G_{\la}^+$'s are at least $(d+1)\de$ on $V_R^+$.
Differentiation of the above identity leads to
\[
d^N \frac{\pa G_{\la_0}^+}{\pa x}(x, y) = \frac{\pa G_{\si^N(\la_0)}^+}{\pa x} (H_{\la_0}^{+N}) \frac{ \pa (\pi_1 \circ H_{\la_0}^{+N}) }{\pa x}(x, y) + \frac{\pa
G_{\si^N(\la_0)}^+}{\pa y}(H_{\la_0}^{+N})  \frac{ \pa (\pi_2 \circ H_{\la_0}^{+N}) }{\pa x}(x, y).
\]
Let the derivatives of $H_{\la}$ be bounded above on $V_R$ by $A_{\la}$ and let $A = \sup A_{\la} < \infty$. It follows that the derivatives of $H_{\la_0}^{+N}$ are bounded above by
$2^{N-1}A^N$ on $V_R$. Hence
\[
\vert d^N \pa G_{\la_0}^+ / \pa x (x, y) \vert \le C (2A)^N.
\]
Let $\ga = \log 2A/ \log d$ so that $C (2A)^N = C d^{N \ga}$. Therefore
\[
\vert \pa G_{\la_0}^+ / \pa x \vert \le C d^{N(\ga - 1)} \le C (d \de/G_{\la_0}^+)^{\ga - 1}
\]
which implies that
\[
\vert \pa (G_{\la_0}^+)^{\ga}/ \pa x \vert \le C \ga(d \de)^{\ga - 1}.
\]
A similar argument can be used to bound the partial derivative of $(G_{\la_0}^+)^{\ga}$ with respect to $y$. Thus the gradient of $(G_{\la_0}^+)^{\ga}$ is bounded uniformly at all
points that are close to $K_{\la_0}^+$.

\medskip

Now suppose that $(x, y) \in S \setminus K_{\la_0}^+$ is such that
\[
d^{N} \de < G_{\la_0}^+(x, y) \le d^{N + 1} \de
\]
for some $N > 0$. This means that $(x, y)$ is far away from $K_{\la_0}^+$ and the above equation can be written as
\[
\de < d^{-N} G_{\la_0}^+(x, y) \le d \de.
\]
By the surjectivity of $\si$, there exists a $\mu_0 \in M$ such that $\si^N(\mu_0) = \la_0$. With this the invariance property of the Green's functions now reads
\[
G_{\mu_0}^+ \circ (H_{\mu_0}^{+N})^{-1}(x, y) = d^{-N} G_{\la_0}^+(x, y).
\]
The compactness of $S$ shows that there is a fixed integer $m < 0$ such that if $(x, y)$ is far away from $S \setminus K_{\la_0}^+$ then it be can brought into the strip
\[
\big\{ (x,y) : \delta < G_{\lambda_0}^+(x,y) \leq d \delta \big\}
\]
by $(H_{\la}^{+ \vert k \vert})^{-1}$ for some $m \le k < 0$ and for all $\la \in M$. By enlarging $R$ we may assume that the image of $S$ under all the maps $(H_{\la}^{+ \vert k
\vert})^{-1}$, $m \le k < 0$ is contained in $V_R$. By increasing $A$, we may also assume that all the derivatives of $H_{\la}$ and $H_{\la}^{-1}$ are bounded by $A$ on $V_R$.
Now repeating the same argument as above, it follows that the gradient of $(G_{\la_0}^+)^{\ga}$ is bounded uniformly at all points that are far away from $K_{\la_0}^+$ -- the nuance
about choosing $\ga$ as before is also valid. The choice of $\mu_0$ such that $\si^{N}(\mu_0) = \la_0$ is irrelevant since the derivatives involved are with respect to $x, y$ only.
The only remaining case is when $(x, y) \in \Om_{\la_0}^{\de}$ which precisely means that $N = 0$. But in this case, $(G_{\la_0}^+)^{\ga - 1}$ is uniformly bounded
on $V_R$ and so are the derivatives of $G_{\la_0}^+$ on $\Om_{\la_0}^{\de}$ by the reasoning given earlier. Therefore there is a uniform bound on the gradient of $(G_{\la_0}^+)^{\ga}$
everywhere on $S$. This shows that $(G_{\la_0}^+)^{\ga}$ is Lipschitz on $S$ which implies that $G_{\la_0}^+$ is H\"{o}lder continuous on $S$ with exponent $1/\ga = \log d/ \log 2A$.
A set of similar arguments can be applied to deduce analogous results for $G_{\la}^{-}$.


\subsection*{Proof of Proposition \ref{pr2}}

We have
\[
(H_{\la}^{\pm 1})^{\ast}(\mu_{\sigma(\lambda)}^\pm) = (H_{\la}^{\pm 1})^{\ast}(dd^cG_{\sigma(\lambda)}^\pm) = dd^c(G_{\sigma(\lambda)}^\pm \circ H_\lambda^{\pm 1}) = dd^c(d^{\pm
1}G_\lambda^\pm) = d^{\pm 1}\mu_\lambda^\pm
\]
where the third equality follows from Proposition \ref{pr1}. A similar exercise shows that
\[
(H_{\la}^{\pm 1})_{\ast} \mu_{\la}^{\pm} = d^{\mp 1} \mu_{\si(\la)}^{\pm}.
\]
If $\si$ is the identity on $M$, then
\[
G_{\la}^+ \circ H_{\la}^{\pm 1} = d^{\pm 1} G_{\la}^{+} \; \text{and} \; G_{\la}^{-} \circ H_{\la}^{\pm 1} = d^{\mp 1} G_{\la}^{-}
\]
which in turn imply that
\[
(H_{\la}^{\pm 1})^{\ast} \mu_{\la} = (H_{\la}^{\pm 1})^{\ast} (\mu_{\la}^+ \wedge \mu_{\la}^-) = (H_{\la}^{\pm 1})^{\ast} \mu_{\la}^+ \wedge (H_{\la}^{\pm 1})^{\ast} \mu_{\la}^- =
d^{\pm 1} \mu_{\la}^+ \wedge d^{\mp 1} \mu_{\la}^- = \mu_{\la}.
\]

\medskip

By Proposition \ref{pr1}, the support of $\mu_\lambda^+$ is contained in $J_\lambda^+$. To prove the converse, let $z_0\in J_\lambda^+$ and suppose that $\mu_\lambda^+ =0$ on a
neighbourhood $U_{z_0}$ of $z_0$. This means that $G_\lambda^+$ is pluriharmonic on $U_{z_0}$ and $G_\lambda^+$ attains its minimum value of zero at $z_0$. This implies
that $G_\lambda^+ \equiv 0$ on $U_{z_0}$ which contradicts the fact that $G_\lambda^+>0$ on $\mathbb{C}^2\setminus K_\lambda^+$. Similar arguments can be applied to prove that
supp$(\mu_\lambda^-)=J_\lambda^-$.

 \medskip

Finally, to show that $\la \mapsto J_{\la}^+$ is lower semicontinuous, fix $\la_0 \in M$ and $\ep > 0$. Let $x_0\in
J_{\lambda_0}^+= {\rm supp}(\mu_{\lambda_0}^+)$. Then $\mu_{\lambda_0}^+(B(x_0, {\epsilon}/{2}))\neq 0$. Since the correspondence $\lambda \mapsto \mu_\lambda^+$ is continuous, there
exists a $\delta>0$ such that
\[
d(\lambda,\lambda_0)<\delta \text{ implies } \mu_\lambda^+(B(x_0;{\epsilon}/{2}))\neq 0.
\]
Therefore $x_0\in {(J_\lambda^+)}^\epsilon=\bigcup_{a\in J_\lambda^+}B(a,\epsilon)$ for all $\lambda \in M$ satisfying $d(\lambda,\lambda_0)< \delta$.
Hence the correspondence $\lambda\mapsto J_\lambda^{\pm}$ is lower semicontinuous.

\medskip

Let $\cal L$ be the class of plurisubharmonic functions on $\mbb C^2$ of logarithmic growth, i.e.,
$$
\mathcal{L}=\{ u\in \mathcal{PSH}(\mathbb{C}^2): u(x,y)\leq \log^+\lVert (x,y) \rVert +L \}
$$
for some $L>0$ and let
$$
\tilde{\mathcal{L}}=\{ u\in \mathcal{PSH}(\mathbb{C}^2):\log^+\lVert (x,y) \rVert -L \leq u(x,y)\leq \log^+\lVert (x,y) \rVert +L\}
$$
for some $L>0$. Note that there exists $L>0$ such that
$$
G_\lambda^+(z)\leq \log^+ \lVert z \rVert +L
$$
for all $z\in \mathbb{C}^2$ and for all $\lambda\in M$. Thus $G_\lambda^+ \in \mathcal{L}$ for all $\lambda\in M$. For $E\subseteq \mathbb{C}^2$, the pluricomplex Green function
of $E$ is
$$
L_E(z)=\sup\{u(z):u\in\mathcal{L},u\leq 0 \text{ on } E\}.
$$
and let $L_E^{\ast}(z)$ be its upper semicontinuous regularization.

\medskip

It turns out that the pluricomplex Green function of $K_\lambda^{\pm}$ is $G_\lambda^{\pm}$ for all $\lambda\in M$. The arguments are similar to those employed for a single H\'{e}non
map and we merely point out the salient features. Fix $\lambda\in M$. Then $G_{\lambda}^+=0$ on $K_\lambda^+$ and  $G_\lambda^+ \in \mathcal{L}$. So $G_\lambda^+ \leq L_{K_{\lambda}^+}$.
To show equality, let $u\in \mathcal{L}$
be such that $u\leq 0=G_\lambda^+$ on $K_\lambda^+$. By Proposition \ref{pr1}, there exists $M>0$ such that
\[
\log\lvert y \rvert-M<G_\lambda^+(x,y)<\log\lvert y \rvert+M
\]
for $(x,y)\in V_R^+$. Since $u\in \mathcal{L}$,
\[
u(x,y)-G_\lambda^+(x,y)\leq M_1
\]
for some $M_1 > 0$ and $(x,y)\in V_R^+.$

Fix $x_0 \in\mbb C$ and note that $u(x_0,y)-G_\lambda^+(x_0,y)$ is a bounded subharmonic function on the vertical line $T_{x_0}=\mathbb{C}\setminus (K_\lambda^+ \cap \{x=x_0\})$ and
hence it can be extended across the point $y=\infty$ as a subharmonic function. Note also that
$$
u(x_0,y)-G_\lambda(x_0,y)\leq 0
$$
on $\partial T \subseteq K_\lambda^+ \cap \{x=x_0\}$. By the maximum principle it follows that $u(x_0,y)-G_\lambda(x_0,y)\leq 0$ on $T_{x_0}$. This implies that $u\leq G_\lambda^+
\text{ in } \mathbb{C}^2\setminus K_\lambda^+$ which in turn shows that
$L_{K_{\lambda}^{+}}=G_{\lambda}^{+}$. Since $G_\lambda^+$ is continuous on $\mathbb{C}^2$, we have
\[
L_{K_{\lambda}^{+}}=L^{\ast}_{K_{\lambda}^{+}}=G_\lambda^+.
\]
Similar arguments show that
\[
L_{K_{\lambda}^{-}}=L^{\ast}_{K_{\lambda}^{-}}=G_{\lambda}^{-}.
\]
Let $u_\lambda=\max \{G_\lambda^+,G_\lambda^-\}$. Again by Proposition \ref{pr1} it follows that $u_\lambda\in \tilde{\mathcal{L}}$.
For $\epsilon>0$, set $G_{\lambda,\epsilon}^{\pm}=\max \{G_\lambda^{\pm},\epsilon\}$ and $u_{\lambda,\epsilon}=\max \{G_{\lambda,\epsilon}^+, G_{\lambda,\epsilon}^{-}\}$.
By Bedford--Taylor,
\[
{(dd^c u_{\lambda,\epsilon})}^2=dd^c G_{\lambda,\epsilon}^+ \wedge dd^c G_{\lambda,\epsilon}^{-}.
\]
Now for a $z\in \mathbb{C}^2 \setminus K_\lambda^{\pm}$ ,
there exists a small neighborhood $\Omega_{z}\subset \mathbb{C}^2\setminus K_\lambda^{\pm}$ of $z$ such that
${(dd^c u_{\lambda,\epsilon})}^2=0$ on $\Omega_z$ for sufficiently small $\epsilon$. It follows that supp${((dd^c u_\lambda))}^2 \subset K_\lambda$.


\medskip

Since $G_\lambda^{\pm}=L^{\ast}_{{K_\lambda}^{\pm}} \leq L^{\ast}_{K_\lambda}$, we have $u_\lambda\leq L^{\ast}_{K_\lambda}$. Further note that $L^{\ast}_{K_\lambda} \leq
L_{K_\lambda}\leq 0=u_\lambda$ almost every where on $K_\lambda$ with respect to the measure ${(dd^c u_\lambda )}^2$. This is because the set
$\{L_{K_\lambda}^* > L_{K_\lambda}\}$ is pluripolar and consequently has measure zero with respect to ${(dd^c u_\lambda)}^2$. Therefore $L^{\ast}_{K_\lambda}\leq u_\lambda$ in
$\mathbb{C}^2$. Finally, $L_{K_\lambda}$ is continuous and thus $L^{\ast}_{K_\lambda}=L_{K_\lambda}=\max \{G_\lambda^+, G_\lambda^-\}$.

For a non-pluripolar bounded set $E$ in $\mathbb{C}^2$ the complex equilibrium measure is $\mu_E={(dd^c L^{\ast}_E)^2}$. Again by
Bedford--Taylor, $\mu_{K_\lambda}=
\lim_{\epsilon \ra 0}{(dd^c \max\{G_\lambda^+, G_\lambda^-,\epsilon\})}^2$ which when combined with
$$
\mu_\lambda=\mu_\lambda^+ \wedge \mu_\lambda^-= \lim_{\epsilon\ra 0}dd^c G_{\lambda,\epsilon}^+ \wedge dd^c G_{\lambda,\epsilon}^-
$$
and
$$
{(dd^c \max \{G_\lambda^+, G_\lambda^-,\epsilon\})}^2=dd^c G_{\lambda,\epsilon}^+ \wedge dd^c G_{\lambda,\epsilon}^-
$$
shows that $\mu_\lambda$ is the equilibrium measure of $K_\lambda$. Since supp$(\mu_\lambda^{\pm})=J_\lambda^{\pm}$, we have supp$(\mu_\lambda) \subset J_\lambda$.


\subsection{Proof of Theorem \ref{thm1}}

Let $\cal L_y$ be the subclass of $\cal L$ consisting of all those functions $v$ for which there exists $R > 0$ such that
\[
v(x, y) - \log \vert y \vert
\]
is a bounded pluriharmonic function on $V_R^+$.

\medskip

Fix $\lambda\in M$ and let $\omega= 1/4 \;dd^c  \log (1 + \Vert z \Vert^2)$. For a $(1, 1)$ test form $\varphi$ on $\mbb C^2$, it follows that  there exists a $C >0$ such that
\[
-C \Vert
\varphi \Vert \omega \leq \varphi \le C \Vert \varphi \Vert \omega
\]
by the positivity of $\om$.

\medskip

{\it Step 1:} $S_{\la}$ is nonempty.\\

Note that
\begin{eqnarray}
\frac{1}{d^n}\left |
\int_{\mathbb{C}^2}(H_\lambda^{+n})^{\ast}(\psi T)\wedge \varphi
\right|
&\lesssim &\frac{\Vert \varphi \Vert}{d^n}\int_{\mathbb{C}^2}(H_\lambda^{+n})^{\ast}(\psi T)\wedge dd^c \log (1 + \Vert z \Vert^2) \nonumber \\
& \lesssim & \frac{\Vert \varphi \Vert}{d^n}\int_{\mathbb{C}^2}dd^c(\psi T)\wedge \log (1 + \Vert (H_{\la}^{+n})^{-1}(z) \Vert ).\label{13}
\end{eqnarray}

\medskip

Direct calculations show that
\[
\frac{1}{d^n}\log^+ \| (H_\lambda^{+n})^{-1}(z) \| \leq \log^+|z|+C \label{14}
\]
for some $C>0$, for all $n\geq 1$, $\lambda\in M$ and
\begin{equation}
\log (1 + \Vert z \Vert^2) \leq 2 \log^+|z|+2\log 2.\label{15}
\end{equation}
It follows that
\begin{equation*}
0 \le \frac{1}{d^n} \log \left( 1 + \Vert (H_{\la}^{+n})^{-1} \Vert \right) \le 2 \log^+ \vert z \vert + C
\end{equation*}
for some $C>0$, for all $n>0$ and $\lambda\in M$. Hence
\begin{equation}
\frac{1}{d^n}\left | \int_{\mathbb{C}^2} (H_\lambda^{+n})^{\ast}(\psi T)\wedge \varphi \right| \lesssim \Vert \varphi \Vert. \label{16}
\end{equation}

\medskip

The Banach-Alaoglu theorem shows that there is a subsequence $\frac{1}{d^{n_j^\lambda}}(H_{\la}^{+n_j^{\la}})^{\ast}(\psi T)$ that converges in the sense of currents to a positive $(1,1)$ current, say
$\gamma_\lambda$. This shows that $S_\lambda$ is nonempty. It also follows from the above discussion that $\int_{\mbb C^2} \gamma_{\la} \wedge \om < + \infty$.

\medskip

\no {\it Step 2:} Each $\gamma_{\la} \in S_{\la}$ is closed. Further, the support of $\ga_{\la}$ is contained in $K_{\la}^+$.\\

Let $\chi$ be a smooth real $1$-form with compact support in $\mbb C^2$ and let $\psi_1 \ge 0$ be such that $\psi_1 = 1$ in a neighbourhood of ${\rm supp}(\psi)$. Then
\[
\frac{1}{d^{n_j^\lambda}}\int_{\mathbb{C}^2}d \chi \wedge (H_\lambda^{+n_j^{\lambda}})^{\ast}(\psi T) =
\frac{1}{d^{n_j^\lambda}}\int_{\mathbb{C}^2}\chi \circ (H_\lambda^{+n_j^{\la}})^{-1} \wedge d\psi \wedge \psi_1 T.
\]
to obtain which the assumption that ${\rm supp}(\psi) \cap {\rm supp}(dT) = \phi$ is used. By the Cauchy-Schwarz inequality it follows that the term on the right above is dominated by
the square root of
\[
\left(\int_{\mathbb{C}^2} \big( (J \chi \wedge \chi)\circ (H_\lambda^{+n_j^{\la}})^{-1} \big) \wedge \psi_1 T\right)
\left( \int_{\mathbb{C}^2} d\psi \wedge d^c \psi \wedge \psi_1 T \right)
\]
whose absolute value in turn is bounded above by a harmless constant times $d^{n_j^\lambda}$. Here $J$ is the standard $\mbb R$-linear map on $1$-forms satisfying $J(d z_j) = i d \ov
z_j$ for $j = 1, 2$.
Therefore
\[
\left| \frac{1}{d^{n_j^{\la}}}
\int_{\mathbb{C}^2}(\chi \circ (H_{\la}^{+n_j^{\la}})^{-1} \wedge d\psi \wedge \psi_1 T
\right| \lesssim d^{- n_j^{\la} / 2}.
\]
Evidently, the right hand side tends to zero as $j \ra \infty$. This shows that $\gamma_\lambda$ is closed.

\medskip

Let $R>0$ be large enough so that supp$(\psi T)\cap V_R^+=\phi$. Let $z\notin K_\lambda^+$ and $B_z$ a small open ball around it such that
$\overline{B_z} \cap K_\lambda^+=\phi$. By Lemma \ref{le1}, there exists an $N>0$ such that $H_\lambda^{+n}(B_z)\subset V_R^+$ for all $n>N$. Therefore
$B_z \cap \text{supp}(H_\lambda^{+n})^{\ast}(\psi T)= B_z \cap (H_{\la}^{+n})^{-1}(\text{supp}(\psi T))=\phi$ for all $n>N$. Since supp$(\gamma_\lambda)\subset
\overline{\bigcup_{n=N}^\infty \text{supp}(H_\lambda^{+n})^{\ast}(\psi T)})$, we have $z\notin \text{supp}(\gamma_\lambda)$. This implies $\text{supp}(\gamma_\lambda)\subset
K_\lambda^+$. Since $K_\lambda^+\cap V_R^+=\phi$ for all $\lambda\in M$, it also follows that $\text{supp}(\gamma_\lambda)$ does not intersect $\ov{V_R^+}$.

\medskip

\no {\it Step 3:} Each $\ga_{\la}$ is a multiple of $\mu_{\la}^+$.

\medskip

It follows from Proposition 8.3.6 in \cite{MNTU} that $\ga_{\la} = c_{\ga, \la} dd^c U_{\ga, \la}$ for some $c_{\ga, \la} > 0$ and $U_{\ga, \la} \in \cal L_y$. In this
representation, $c_{\ga, \la}$ is unique while $U_{\ga, \la}$ is unique upto additive constants. We impose the following condition on $U_{\gamma,\lambda}$:
\[
\lim_{|y|\rightarrow \infty} (U_{\gamma,\lambda}-\log|y|)=0 \label{17}
\]
and this uniquely determines $U_{\ga, \la}$. It will suffice to show that $U_{\ga, \la} = G_{\la}^+$.

\medskip

Let $\gamma_{\lambda,x}$ denote the restriction of $\gamma_\lambda$ to the plane $\{(x,y):y\in \mathbb{C}\}$. Since $U_{\ga, \la} \in \cal L_y$, it follows that

\begin{equation}
\int_{\mathbb{C}}\gamma_{\lambda,x}=2\pi c_{\gamma,\lambda}, \;\; U_{\gamma,\lambda}(x,y)=\frac{1}{2\pi c_{\gamma,\lambda}} \int_{\mathbb{C}}\log |y-\zeta|\gamma_{\lambda,x}(\zeta).
\label{18}
\end{equation}

\medskip

Consider a uniform filtration $V^{\pm}_R, V_R$ for all the maps $H_{\la}$ where $R^d > 2R$ and $\vert p_{j, \la}(y) \vert \ge \vert y \vert^d / 2$ for $\vert y \vert \ge R$. Let $0
\not= a = \sup \vert a_j(\la) \vert < \infty$ (where the supremum is taken over all $1 \le j \le m$ and $\la \in M$) and choose $R_1 > R^d /2$. Define
\[
A = \big \{ (x, y) \in \mbb C^2 : \vert y \vert^d \ge 2(1 + a) \vert x \vert + 2 R_1 \big \}.
\]
Evidently $A \subset \{ \vert y \vert > R \}$. Lemma \ref{le1} shows that for all $\la \in M$, $H_{\la}(x, y) \subset V_R^+$ when $(x, y) \in A \cap V_R^+$. Furthermore for $(x, y) \in
A \cap (\mbb C^2 \setminus V_R^+)$, it follows that
\[
\vert p_{j, \la}(y) - a_j(\la)x \vert \ge \vert y \vert^d / 2 - a \vert x \vert \ge \vert y \vert + R.
\]
This shows that $H_{\la}(A) \subset V_R^+$. By Lemma \ref{le1} again it can be seen that $H_{\la}^{+n}(A) \subset V_R^+$ for all $n \ge 1$ which shows that $A \cap K_{\la}^+ = \phi$ for
all $\la \in M$. Let $C>0$ be such that
\[
C^d \geq \max\{2(1+\lvert a \rvert), 2R_1\}.
\]
If $|y|\geq C(\lvert x \rvert^{1/d}+1)$ then
\begin{equation*}
{|y|}^{d} \geq C^{d}(\lvert x\rvert+1) \geq 2(1+\lvert a \rvert)\lvert x \rvert + 2R_1.
\end{equation*}
which implies that
\[
B= \big\{ (x,y)\in \mathbb{C}^2: |y|\geq C(\lvert x \rvert^{1/d}+1) \big\}\subset A
\]
and hence $K_\lambda^+ \cap B=\phi $. Since $V_R^+ \subset B $ for sufficiently large $R$, by applying Lemma \ref{le1} once again it follows that
\begin{equation}
K_\lambda^+\cap B=\phi \text{ and } \bigcup_{n=0}^\infty (H_{\la}^{+n})^{-1}(B)=\mathbb{C}^2\setminus K_\lambda^+\label{19}
\end{equation}
for all $\lambda\in M$.

\medskip

Set $r=C(|x|^{1/d}+1)$. Since supp$(\gamma_\lambda)\subset K_\lambda^+$ it follows that
\[
\text{supp}(\gamma_{\lambda,x})\subset \{\lvert y \rvert \leq r\}
\]
for all $\lambda\in M$. Since
\[
\lvert y \rvert-r\leq \lvert y-\zeta\rvert\leq \lvert y \rvert+r
\]
for $\lvert y \rvert>r$ and $\lvert \zeta\rvert\leq r$, (\ref{18}) yields
\[
\log(\lvert y \rvert-r) \leq U_{\gamma,\lambda}(x,y) \leq \log(\lvert y \rvert+r)
\]
which implies that
\[
-(r/{\lvert y \rvert})/(1- r/{\lvert y \rvert})\leq U_{\gamma,\lambda}(x,y)- \log \lvert y \rvert \leq  r/{\lvert y \rvert}.
\]
Hence for $\lvert y \rvert > 2r$, we get
\begin{equation}
-2r/{\lvert y \rvert} \leq U_{\gamma,\lambda}(x,y)- \log \lvert y \rvert \leq r/{\lvert y \rvert} \label{20}
\end{equation}
for all $\lambda\in M$.

\medskip

For each $N \ge 1$, let $\gamma_\lambda(N) = d^{N}(H_{\la}^{+N})_{\ast}(\gamma_\lambda)$. Then
\[
\gamma_\lambda(N)=\lim_{j \ra \infty} d^{-n_j + N}\big( H_{\si^N(\la)}^{+(n_j - N)} \big)^{\ast}(\psi T) \in S_{\sigma^N(\lambda)}(\psi T).
\]
Therefore
\[
\ga_{\si^N(\la)} = c_{\gamma,\sigma^N(\lambda)}dd^c U_{\gamma,\sigma^N(\lambda)}
\]
for some  $c_{\gamma,\sigma^N(\lambda)}>0$ and $U_{\gamma,\sigma^N(\lambda)}\in \mathcal{L}_y$ and moreover
\[
c_{\gamma,\lambda}dd^c U_{\gamma,\lambda} = \ga_{\la} = d^{-N} \big( H_{\la}^{+N}
\big)^{\ast} \ga_{\si^N(\la)} = c_{\gamma,\sigma^N(\lambda)}dd^c \big(d^{-N} \big(H_{\la}^{+N} \big)^{\ast} U_{\ga, \si^N(\la)} \big).
\]
Note that both $d^{-N} \big(H_{\la}^{+N} \big)^{\ast} U_{\ga, \si^N(\la)}$ and $U_{\gamma,\sigma^N(\lambda)}$ belong to $\mathcal{L}_y$. It follows that $c_{\ga,
\la} = c_{\ga, \si^N(\la)}$ and $d^{-N} \big(H_{\la}^{+N} \big)^{\ast} U_{\ga, \si^N(\la)}$ and $U_{\gamma,\lambda}$ coincide up to an additive constant which can be shown to be
zero as follows.

\medskip

By the definition of the class $\cal L_y$, there exists a pluriharmonic function $u_{\la, N}$ on some $V_R^+$ such that
\[
U_{\gamma,\sigma^N(\lambda)}(x,y)- \log \lvert y \rvert = u_{\lambda,N} \text{ and } \lim_{\lvert y \rvert \rightarrow \infty}u_{\lambda,N}(x,y)= u_0 \in \mathbb{C}.
\]
Therefore if $(x,y)\in (H_{\la}^{+N})^{-1}(V_R^+)$ and $(x_N^{\la}, y_N^{\la}) = H_{\la}^{+N}(x, y)$ then
\[
d^{-N} \big(H_{\la}^{+N} \big)^{\ast} U_{\ga, \si^N(\la)} (x, y) - d^{-N}\log \lvert y_N^\lambda \rvert = d^{-N}u_{\lambda, N}(x_N^\lambda,y_N^\lambda).
\]
By (2.15), we have that
\[
d^{-N}\log\lvert y_N^\lambda\rvert - \log\lvert y \rvert \rightarrow 0
\]
as $\lvert y \rvert \ra \infty$ which shows that
\[
d^{-N} \big(H_{\la}^{+N} \big)^{\ast} U_{\ga, \si^N(\la)}(x, y) - \log\lvert y \rvert \rightarrow 0
\]
as $\vert y \vert \ra \infty$. But by definition
\[
U_{\gamma,\lambda}(x,y) - \log\lvert y \rvert \rightarrow 0
\]
as $\lvert y \rvert\rightarrow \infty$ and this shows that $ d^{-N} \big(H_{\la}^{+N} \big)^{\ast} U_{\ga, \si^N(\la)} = U_{\gamma,\lambda}$.

\medskip

Let $(x,y)\in \mathbb{C}^2\setminus K_\lambda^+$ and $\epsilon>0$. For a sufficiently large $n$, $(x_n^\lambda, y_n^\lambda)=H_\lambda^{+n}(x,y)$ satisfies $\lvert x_n^\lambda\rvert
\leq \lvert y_n^\lambda\rvert$ and $(x_n^\lambda,y_n^\lambda)\in B$ as defined above. Hence by (\ref{20}) we get
\[
\left| d^{-n} \big(H_{\la}^{+n} \big)^{\ast} U_{\ga, \si^n(\la)} - d^{-n}\log \lvert y_n^\lambda \rvert \right| \leq \frac{2C}{d^n\lvert y _n^\lambda\rvert}({\lvert x_n^\lambda
\rvert}^{1/d}+1)<\epsilon.
\]
On the other hand by using (\ref{9.1}), it follows that
\[
\left| G_\lambda^+(x,y)- d^{-n}\log\lvert y_n^\lambda\rvert\right|<\epsilon
\]
for large $n$. Combining these two inequalities and the fact that $ d^{-n} \big(H_{\la}^{+n} \big)^{\ast} U_{\ga, \si^n(\la)}=U_{\gamma,\lambda}$ for all $n\geq 1$ we get
\[
\left| G_\lambda^+(z)-U_{\gamma,\lambda}(z)\right|<2\epsilon.
\]
Hence $U_{\gamma,\lambda}=G_\lambda^+$ in $\mathbb{C}^2\setminus K_\lambda^+$.

\medskip

The next step is to show that $U_{\gamma,\lambda}=0$ in the interior of $K_\lambda^+$. Since $U_{\gamma,\lambda}=G_\lambda^+$ in $\mathbb{C}^2\setminus K_\lambda^+$, the maximum
principle applied to $U_{\gamma,\lambda}(x,.)$ with $x$ being fixed, gives $U_{\gamma,\lambda}\leq 0$ on $K_\lambda^+$. Suppose that there exists a nonempty
$\Omega\subset\subset K_\lambda^+$ satisfying
$U_{\gamma,\lambda}\leq -t$ in $\Omega$ with $t>0$. Let $R>0$ be so large that $\bigcup_{n=0}^{\infty}H_\lambda^{+n}(\Omega)\subset V_R$ -- this follows from Lemma \ref{le1}. Since
$d^{-n} \big(H_{\la}^{+n} \big)^{\ast} U_{\ga, \si^n(\la)}=U_{\gamma,\lambda}$ for each $n \ge 1$, it follows that
\[
H_\lambda^{+n}(\Omega)\subset \big\{U_{\gamma,\sigma^n(\lambda)}\leq -d^n t\big\}\cap V_R
\]
for each $n \ge 1$. The measure of the last set with $x$ fixed and $\lvert x \rvert\leq R$ can be estimated in this way -- let
\[
Y_x=\big\{ y\in \mathbb{C}:U_{\gamma,\sigma^n(\lambda)}\leq -d^n t\big\}\cap \big\{\lvert y \rvert <R\big\}.
\]
By the definition of capacity
\[
\text{cap}(Y_x)\leq \exp (-d^n t)
\]
and since the Lebesgue measure of $Y_x$, say $m(Y_x)$ is at most $\pi e {\text{cap}(Y_x)}^2$ (by the compactness of $Y_x \subset \mbb C$) we get
\[
m(Y_x)\leq \pi \exp(1-2d^n t).
\]
Now for each $\la\in M$, the Jacobian determinant of $H_\la$ is a constant given by $a_\la= a_1(\la) a_2(\la) \ldots a_m(\la)\neq 0$ and since the correspondence $\la \mapsto a_\la$ is
continuous, an application of Fubini's theorem yields
\[
a^n m(\Omega)\leq \lvert a_{\sigma^{n-1}(\la)}\cdots a_\la\rvert m(\Omega)=m(H_\la^{+n}(\Omega))\leq \int_{\lvert x \rvert\leq R}m(Y_x)dv_x \leq \pi^2 R^2 \exp (1-2d^n t)
\]
where $a=\inf_{\la\in M} \lvert a_\la \rvert $. This is evidently a contradiction for large $n$ if $m(\Omega)>0$.

\medskip

So far it has been shown that $U_{\gamma,\lambda}=G_\lambda^+$ in $\mathbb{C}^2\setminus J_\lambda^+$. By using the continuity of $G_\lambda^+$ and the upper semi-continuity of
$U_{\gamma,\lambda}$, we have that $U_{\gamma,\lambda}\geq G_\lambda^+$ in $\mathbb{C}^2$. Let $\epsilon>0$ and consider the slice $D_\lambda=\{y:G_\lambda^+<\epsilon\}$ in the
$y$-plane for some fixed $x$. Note that $U_{\gamma,\lambda}(x,.)=G_\lambda^+(x,.)=\epsilon$ on the boundary $\partial D$. Hence by the maximum principle $U_{\gamma,\lambda}(x,.)\leq
\epsilon$ in $D_\lambda$. Since $x$ and $\epsilon$ are arbitrary, it follows that $U_{\gamma,\lambda}^+=G_\lambda^+$ in $\mathbb{C}^2$. This implies that
\[
\gamma_\lambda=c_{\gamma,\lambda}\mu_\lambda^+
\]
for any $\gamma_\lambda\in S_\lambda(\psi T)$. This completes the proof of Theorem 1.3.


\subsection{Proof of Proposition 1.4}

Let $\si : M \ra M$ be an arbitrary continuous map and pick a $\ga_{\la} \in S(\psi, T)$. Let $\theta = 1/2 \;dd^c \log (1 + \vert x \vert^2)$ in $\mbb C^2$ (with coordinates $x, y$)
which is a positive closed $(1, 1)$-current depending only on $x$. Then for any test function $\varphi$ on $\mbb C^2$,
\[
\int_{\mathbb{C}^2}\varphi \gamma_\lambda \wedge \theta = c_{\gamma,\lambda}\int_{\mathbb{C}^2}U_{\gamma,\lambda}dd^c \varphi \wedge \theta
= c_{\gamma,\lambda}\int_{\mathbb{C}}\theta \int_{\mathbb{C}}U_{\gamma,\lambda} \Delta_y \varphi
= c_{\gamma,\lambda} \int_{\mathbb{C}}\theta \int_{\mathbb{C}}\varphi \Delta_y U_{\gamma,\lambda}.
\]
Since $y \mapsto U_{\gamma,\lambda}(x,y)$ has logarithmic growth near infinity and $\varphi$ is arbitrary it follows that
\begin{equation}
\int_{\mathbb{C}^2}\gamma_\lambda \wedge \theta = 2\pi c_{\gamma,\lambda}\int_{\mathbb{C}^2}\theta
={(2\pi)}^2c_{\gamma,\lambda}.\label{19}
\end{equation}
Let $R > 0$ be large enough so that $\text{supp}(\psi T) \cap V_R^+ = \phi$ which implies that $\text{supp}(\psi T)$ is contained in the closure of $V_R \cup V_R^-$. Then
\begin{eqnarray*}
\int_{\mathbb{C}^2}\frac{1}{d^{n_j^\lambda}} (H_\lambda^{{+ n_j^\lambda}})^{\ast}(\psi T) \wedge \theta
&=& \frac{1}{d^{n_j^\lambda}}\int_{\mathbb{C}^2}\psi T \wedge \frac{1}{2} (H_\lambda^{{+ n_j^\lambda}})_{\ast}dd^c\log (1+|x|^2)\\
&=& \frac{1}{d^{n_j^\lambda}}\int_{\mathbb{C}^2} (\psi T)\wedge dd^c\left(\frac{1}{2}\log (1+|\pi_1\circ (H_\lambda^{{+ n_j^\lambda}})^{-1}|^2)\right)\\
&=& \frac{1}{d^{n_j^\lambda}} \int_{\overline{V_R\cup V_R^-}}\psi T \wedge dd^c\left(\frac{1}{2}\log (1+|\pi_1\circ (H_\lambda^{+ n_j^\lambda})^{-1}|^2)\right).
\end{eqnarray*}
It is therefore sufficient to study the behavior of $\log (1+|\pi_1\circ (H_\lambda^{+ n_j^\lambda})^{-1}|^2)$. But
\[
\log^+ \vert x \vert \le \log^+ \vert (x, y) \vert \le \log^+ \vert x \vert + R
\]
for $(x, y) \in V_R \cup V_R^-$ and by combining this with
\[
2 \log^+ \vert x \vert \le \log (1 + \vert x \vert^2) \le 2 \log^+ \vert x \vert + \log 2
\]
it follows that the behavior of $(1/2) d^{-n_j^{\la}} \log (1+|\pi_1\circ (H_\lambda^{+ n_j^\lambda})^{-1}|^2)$ as $j \ra \infty$ is similar to that of
$d^{-n_j^{\la}} \log^+ \vert (H_\lambda^{+ n_j^\lambda})^{-1} \vert$.

\medskip

Now suppose that $\si$ is the identity on $M$. In this case, $(H_\lambda^{+ n_j^\lambda})^{-1}$ is just the usual  $n_j^{\la}$--fold iterate of the map $H_{\la}$ and by Proposition 1.1
it follows that
\[
\lim_{j \ra \infty} d^{-n_j^{\la}} \log \Vert (H_\lambda^{+n_j^\lambda})^{-1} \Vert = G_{\la}^-
\]
and hence that
\[
4 \pi^2 c_{\ga, \la} = \int_{\mbb C^2} \ga_{\la} \wedge \theta = \int_{\mathbb{C}^2} \lim_{j \ra \infty} \frac{1}{d^{n_j^\lambda}} (H_\lambda^{{+ n_j^\lambda}})^{\ast}(\psi T) \wedge
\theta = \int_{\mbb C^2} \psi T \wedge \mu_{\la}^-.
\]
The right side is independent of the subsequence used in the construction of $\ga_{\la}$ and hence $S(\psi, T)$ contains a unique element.

\medskip

The other case to consider is when there exists a $\la_0 \in M$ such that $\si^n(\la) \ra \la_0$ for all $\la$. For each $n \ge 1$ let
\[
\ti G_{n, \la}^- = \frac{1}{d^n} \log^+ \Vert (H_{\la}^n)^{-1} \Vert.
\]
Note that $\ti G_{n, \la}^-  \not= G_{n, \la}^-!$ It will suffice to show that $\ti G_{n, \la}^-$ converges uniformly on compact subsets of $\mbb C^2$ to a plurisubharmonic function,
say
$\ti G_{\la}^-$. Let
\[
\ti K_{\la}^- = \big\{ z \in \mbb C^2 : \;\text{the sequence} \;\{ (H_{\la}^{+n})^{-1}(z) \}  \;\text{is bounded} \;\big\}
\]
and let $A \subset \mbb C^2$ be a relatively compact set such that $A \cap \ti K_{\la}^- = \phi$ for all $\la \in M$. The arguments used in Lemma \ref{le1} show that
\[
\mbb C^2 \setminus \ti K_{\la}^- = \bigcup_{n=0}^{\infty} H_{\la}^{+n}(V_R^-)
\]
for a sufficiently large $R > 0$. As Proposition 1.1 it can be shown that $\ti G_{n, \la}^-$ converges to a pluriharmonic function $\ti G_{\la}^-$ on $V_R^-$. Hence for large $m, n$
\begin{equation}
\vert \ti G_{m, \la}^-(p) - \ti G_{n, \la}^-(q) \vert < \ep
\end{equation}
for $p, q \in V_R^-$ that are close enough. Let $n_0$ be such that $(H_{\la_0}^{+n_0})^{-1}(A) \subset V_R^-$ and pick a relatively compact set $S \subset V_R^-$
such that $(H_{\la_0}^{+n_0})^{-1}(A) \subset S$. Pick any $\la$. Since $\si^n(\la) \ra \la_0$ and the maps $H_{\la}^{\pm 1}$ depend continuously on $\la$, it follows that
$H_{\si^n(\la)}^{+n_0}(A) \subset S$. By choosing $m, n$ large enough it is possible to ensure that for all $(x, y) \in A$, $(H_{\si^{m - n_0}(\la)}^{+n_0})^{-1}(x, y)$ and
$(H_{\si^{n - n_0}(\la)}^{+n_0})^{-1}(x, y)$ are as close to each other as desired. By writing
\[
\ti G_{n, \la}^-(x, y) = \frac{1}{d^{n_0}} \frac{1}{d^{n - n_0}} \log^+ \Vert H_{\la}^{-1} \circ \cdots \circ H_{\si^{n - n_0 + 1}(\la)}^{-1} \circ
(H_{\si^{n - n_0}(\la)}^{+n_0})^{-1}(x, y) \Vert
\]
and using (2.25) it follows that $\ti G_{n, \la}^-$ converges uniformly to a pluriharmonic function on $A$. To conclude that this convergence is actually uniform on compact sets of
$\mbb C^2$,
it suffices to appeal to the arguments used in Proposition 1.1.


\subsection{Proof of Theorem 1.5}

Recall that now $\si$ is the identity and
\begin{equation}
H(\la, x, y) = (\la, H_{\la}(x, y)).
\end{equation}
Thus the second coordinate of the $n$-fold iterate of $H$ is simply the $n$-fold iterate $H_{\la} \circ
H_{\la} \circ \cdots \circ H_{\la}(x, y)$. For simplicity, this will be denoted by $H_{\la}^n$ as opposed to $H_{\la}^{+n}$ since they both represent the same map. Consider the disc
$\mathcal{D}= \{x=0,\vert y \vert < R\} \subset \mbb C^2$ and let $0 \le \psi \le 1$ be a test function with compact support in $\cal D$ such that $\psi \equiv 1$ in a $\cal D_r = \{x =
0, \vert y \vert < r\}$ where $r < R$. Let $\imath :\mathcal{D}\ra V_R$ be the inclusion map. Let $L$  be a smooth subharmonic function of $\vert y \vert$ on the
$y$-plane such that $L(y)=\log \vert y \vert$ for $\vert y \vert > R$ and define $\Theta= (1/2\pi) dd^c L$. If $\pi_y$ be the projection from $\mbb C^2$ onto the $y$-axis, let
\[
\alpha_{n,\la}= (\pi_y \circ H_{\la}^n \circ \imath)^{\ast} \Theta \big|_{\cal D_r}.
\]
By using Theorem 1.3 and Proposition 1.4 along with Lemma 4.1 in \cite{BS3} it follows that if $j_n$ be a sequence such that $1 \le j_n < n$ and both $j_n, n - j_n \ra \infty$ then
\[
\lim_{n \ra \infty} d^{-n} (H_{\la}^{j_n})_{\ast} \alpha_{n, \la} = c_{\la} \mu_{\la}
\]
where $c_{\la} = \int \psi [\cal D] \wedge \mu_{\la}^+$. Note that $c_{\la} = 1$ for all $\la \in M$ since $\mu_{\la}^+ = (1/2\pi) dd^c G_{\la}^+$ and $G_{\la}^+ = \log \vert y \vert$
plus a harmonic term in $V_R^+$. As a consequence, if $\si_{n, \la} = d^{-n} \al_{n, \la}$ and
\begin{equation*}
\mu_{n,\la}=\frac{1}{n}\sum_{j=0}^{n-1} (H_\la^j)_{\ast}(\si_{n,\la}), \label{21}
\end{equation*}
then Lemma 4.2 in \cite{BS3} shows that
\begin{equation*}
\lim_{n\ra \infty}\mu_{n,\la}=\mu_\la \label{22}
\end{equation*}
for each $\la\in M$.

\medskip

For an arbitrary compactly supported probability measure $\mu'$ on $M$ and for each $n \ge 0$ let $\mu_n$ and $\sigma_n$ be defined by the recipe in (1.2), i.e., for a test function
$\phi$,
\[
\langle \mu_n, \phi \rangle = \int_M \left ( \int_{\{ \la \} \times \mbb C^2} \phi \; \mu_{n, \la}  \right) \mu'(\la) \;\; \text{and} \;\; \langle \si_n, \phi \rangle  = \int_M \left (
\int_{\{ \la \} \times \mbb C^2} \phi \; \si_{n, \la} \right) \mu'(\la).
\]

\medskip

We claim that
\[
\lim_{n\ra \infty} \mu_n=\mu \; \text{and} \; \mu_n=\frac{1}{n}\sum_{j=0}^{n-1}H_*^j\sigma_n.
\]
where $H$ is as in (2.26). For the first claim, note that for all test functions $\phi$
\begin{eqnarray}
 \lim_{n\ra \infty}\langle \mu_n,\phi\rangle &=& \lim_{n\ra \infty}\int_M \langle \mu_{n,\la},\phi\rangle \mu'(\la) = \int_M \lim_{n\ra \infty}\langle \mu_{n,\la},\phi\rangle
\mu'(\la)\\ \notag
                                             &=& \int_M \langle \mu_\la,\phi\rangle \mu'(\la) = \langle\mu,\phi\rangle \label{23}
\end{eqnarray}
where the second equality follows by the dominated convergence theorem. For the second claim, note that
\[
\left \langle \frac{1}{n}\sum_{j-0}^{n-1}H^j_*\sigma_n,\phi \right \rangle = \int_M \left \langle \frac{1}{n}\sum_{j=0}^{n-1}{H_\la^j}_*(\sigma_{n,\la}),\phi \right \rangle \mu'(\la)
                                                                            = \int_M \langle\mu_{n,\la},\phi \rangle \mu'(\la) = \langle \mu_n,\phi\rangle.
\]
Hence by (2.27), we get
\[
\lim_{n\ra\infty}\frac{1}{n}\sum_{j-0}^{n-1}H^j_*\sigma_n=\mu.
\]

\medskip

Note that the support of $\mu$ is contained in $\text{supp}(\mu') \times V_R$. Let $\mathcal{P}$ be a partition of $M\times V_R$ so that the $\mu$-measure of the boundary of  each
element of $\mathcal{P}$ is zero and each of its elements has diameter
less than $\epsilon$. This choice is possible by Lemma 8.5 in \cite{W}. For each $n\geq 0$, define the $d_n$ metric on $M\times V_R$ by
$$
d_n(p,q)=\max_{0\leq i \leq {n-1}}d(H^i(p),H^i(q))
$$
where $d$ is the product metric on $M\times V_R$. Note that each element $\mathcal{B}$ of
$\bigvee_{j=0}^{n-1}H^{-j}\mathcal{P}$ is inside an $\epsilon$-ball in the $d_n$ metric and if $\cal B_{\la} = (\cal B \times \{ \la \}) \cap V_R$, then the $\sigma_n$ measure of $\cal
B$ is given by
\begin{equation*}
 \sigma_n(\mathcal{B}) = \int_M {\sigma_{n,\la}(\mathcal{B}_\la)}\mu'(\la)
 = \int_M \left ( d^{-n}\int_{{\mathcal{B}_\la}\cap \mathcal{D}} {H_\la^{n}}^{\ast} \Theta \right) \mu'(\la)
 = \int_M \left ( d^{-n}\int_{H_\la^n({\mathcal{B}_\la}\cap \mathcal{D})} \Theta \right ) \mu'(\la).
\end{equation*}
Therefore, since $\Theta$ is bounded above on $\mathbb{C}^2$, there exists $C>0$ such that
\begin{equation}
 \sigma_n(\mathcal{B})\leq  C \; d^{-n}\int_M \text{Area}( H_\la^n(\mathcal{B}_\la\cap \mathcal{D})) \mu'(\la)
 = C \; d^{-n} \text{Area} \left( H^n ( \mathcal{B}\cap (\mathcal{D}\times M)) \right).
 \end{equation}

\medskip

For a continuous map $f : X \ra X$ on a compact set $X$ endowed with an invariant probability measure $m$, let
\begin{eqnarray*}
 {\cal H}_m(\mathcal{A}) &=& -{\Sigma_{i=1}^k m({A}_i) \log m({A}_i)},\\
 h(\cal A, f) &=& \lim_{n \ra \infty} \frac{1}{n} \cal H_m \left( \bigvee_{j=0}^{n-1} f^{-j} \cal A \right)
\end{eqnarray*}
for a partition $\mathcal{A}=\{ {A}_1, A_2, \ldots, {A}_k\}$ of $X$. By definition, the measure theoretic entropy of $f$ with respect to $m$ is $h_m(f) = \sup_{\cal A} h(\cal A, f)$. We
will work with $X = \text{supp}(\mu) \subset M \times V_R$ and view $H$ as a self map of $X$.

\medskip

If $v^0(H,n,\epsilon)$ denotes the supremum of the areas of $H^n(\mathcal{B}\cap (\mathcal{D}\times M))$ over all $\epsilon$-balls $\mathcal{B}$, then
\begin{equation*}
\cal H_{\sigma_n}\left( \bigvee_{j=0}^{n-1} H^{-j}\mathcal{P} \right) \geq -{\log C}+n \log d -\log v^0(H,n,\epsilon)
\end{equation*}
by (2.28). By appealing to Misiurewicz's variational principle as explained in \cite{BS3} we get a lower bound for the measure theoretic entropy $h_\mu$ of $H$ with respect to the
measure $\mu$ as follows:
\begin{equation*}
h_\mu \geq \limsup_{n\ra \infty} \frac{1}{n}(-{\log C}+n \log d -\log v^0(H,n,\epsilon)) \geq \log d -\limsup_{n\ra \infty} v^0(H,n,\epsilon).
\end{equation*}
By Yomdin's result (\cite{Y}), it follows that $\lim_{\epsilon\ra 0} v^0(H,n,\epsilon)=0$. Thus $h_\mu\geq \log d$. To conclude, note that $\text{supp}(\mu) \subset \cal J \subset
M\times V_R$ and therefore by the variational principle the topological entropy of $H$ on $\cal J$ is also at least $\log d$.


\section{Fibered families of holomorphic endomorphisms of $\mbb P^k$}

\subsection{Proof of Proposition 1.6}: By (1.4) there exists a $C>1$ such that
\[
C^{-1} \Vert F_{\sigma^{n-1}(\la)} \circ \ldots \circ F_\la(x) \Vert^d \leq \Vert F_{\sigma^{n}(\la)} \circ \ldots \circ F_\la(x)\Vert \leq C \Vert F_{\sigma^{n-1}(\la)}\circ \ldots
\circ F_\la(x)\Vert^d
\]
for all $\la\in M$, $x\in \mathbb{C}^{k+1}$ and for all $n\geq 1$. As a result,
\begin{equation}
\vert G_{n+1,\la}(x)-G_{n,\la}(x) \vert \leq \log C/d^{n+1}. \label{24}
\end{equation}
Hence for each  $\la\in M$,  as $n\ra\infty$, $G_{n,\la}$ converges uniformly to a continuous plurisubharmonic function $G_\la$ on $\mathbb{C}^{k+1}$. If $G_n(\la, x) = G_{n, \la}(x)$,
then (3.1) shows that $G_n \ra G$ uniformly on $M \times (\mbb C^{k+1} \setminus \{0\})$.

\medskip

Furthermore, for $\la\in M$ and $c\in \mathbb{C}^*$
\begin{eqnarray}
G_\la(cx)&=&\lim_{n\ra \infty}\frac{1}{d^n}\log \Vert F_{\sigma^{n-1}(\la)}\circ \ldots \circ F_\la(cx)\Vert \nonumber\\
&=&\lim_{n\ra\infty}\left( \frac{1}{d^n}\log {|c|}^{d^n}+\frac{1}{d^n}\log
\Vert F_{\sigma^{n-1}(\la)}\circ \ldots \circ F_\la(z) \Vert \right) = \log \vert c \vert + G_\la(x).
\end{eqnarray}
We also note that
\[
G_{\sigma(\la)}\circ F_\la(x) = d \lim_{n\ra \infty }\frac{1}{d^{n+1}}\log \Vert F_{\sigma^{n}(\la)}\circ \ldots \circ F_\la(x) \Vert = d G_\la(x)
\]
for each $\la\in M$.

\medskip

Finally, pick $x_0 \in \cal A_{\la_0}$ which by definition means that $\Vert F_{\si^{n-1}(\la_0)} \circ \ldots \circ F_{\si(\la_0)} \circ F_{\la_0}(x_0) \Vert \le \ep$ for all large $n$.
Therefore $G_{n, \la_0}(x_0) \le d^{-n} \log \ep$ and hence $G_{\la_0}(x_0) \le 0$. Suppose that $G_{\la_0}(x_0) = 0$. To obtain a contradiction, note that there exists a uniform $r >
0$ such that
\[
\Vert F_{\la}(x) \Vert \le (1/2) \Vert x \Vert
\]
for all $\la \in M$ and $\Vert x \Vert \le r$. This shows that the ball $B_r$ around the origin is contained in all the basins $\cal A_{\la}$. Now $G_{\la}(0) = -\infty$ for all $\la
\in M$ and since $G_{\la_n} \ra G_{\la}$ locally uniformly on $\mbb C^{k+1} \setminus \{0\}$ as $\la_n\ra \la$ in $M$, it follows that there exists a large $C > 0$ such that
\[
\sup_{(\la, x) \in M \times \pa B_r} G_{\la}(x) \le - C.
\]
By the maximum principle it follows that for all $\la \in M$
\begin{equation}
G_{\la}(x) \le -C
\end{equation}
on $B_r$. On the other hand, the invariance property $G_{\si(\la)} \circ F_{\la} = d G_{\la}$ implies that
\[
d^n G_{\la} = G_{\si^n(\la)} \circ F_{\si^{n-1}(\la)} \circ \ldots \circ F_{\la}
\]
for all $n \ge 1$. Since we are assuming that $G_{\la_0}(x_0) = 0$ it follows that
\[
G_{\si^n(\la_0)} \circ F_{\si^{n-1}(\la_0)} \circ \ldots \circ F_{\la_0}(x_0) = 0
\]
for all $n \ge 1$ as well. But $F_{\si^{n-1}(\la_0)} \circ \ldots \circ F_{\si(\la_0)} \circ F_{\la_0}(x_0)$ is eventually contained in $B_r$ for large $n$ and this means that
\[
0 = G_{\si^n(\la_0)} \circ F_{\si^{n-1}(\la_0)} \circ \ldots \circ F_{\la_0}(x_0) \le -C
\]
by (3.3). This is a contradiction. Thus $\cal A_{\la} \subset \{G_{\la} < 0\}$ for all $\la \in M$.

\medskip

For the other inclusion, let $x \in \mathbb{C}^{k+1}$ be such that $G_\la(x)=-a$ for some $a>0$. This implies that for a given $\epsilon>0$ there exist $j_0$ such that
\[
-(a+\epsilon)< \frac{1}{d^j}\log \Vert F_{\sigma^{j-1}(\la)}\circ \ldots \circ F_\la(x)\Vert < -a+\epsilon
\]
for all $j\geq j_0$. This shows that $ F_{\sigma^{j-1}(\la)}\circ \ldots \circ F_\la(x) \ra 0$ as $j\ra \infty$. Hence $x\in \mathcal{A}_\la$.


\subsection{Proof of Proposition 1.7}: Recall that $\Om_{\la} = \pi(\cal H_{\la})$ where $\cal H_{\la} \subset \mbb C^{k+1}$ is the collection of those points in a neighborhood of
which $G_{\la}$ is pluriharmonic and $\Om'_{\la} \subset \mbb P^k$ consists of those points $z \in \mbb P^k$ in a neighborhood of which the sequence
\[
\{ f_{\si^{n-1}(\la)} \circ \ldots \circ f_{\si(\la)} \circ f_{\la} \}_{n \ge 1}
\]
is normal, i.e., $\Om'_{\la}$ is the Fatou set. Once it is known that the basin $\cal A_{\la} = \{ G_{\la} < 0 \}$, showing
that $\Om_{\la} = \Om'_{\la}$ and that each $\Om_{\la}$ is in fact pseudoconvex and Kobayashi hyperbolic follows in much the same way as in \cite{U}. Here are the main points in the proof:

\medskip

\no {\it Step 1:} For each $\la \in M$, a point $p \in \Om_{\la}$ if and only if there exists a neighborhood $U_{\la, p}$ of $p$ and a holomorphic section $s_{\la} : U_{\la, p} \ra \mbb
C^{k+1}$ such that $s_{\la}(U_{\la, p}) \subset \pa \cal A_{\la}$. The choice of such a section $s_{\la}$ is unique upto a constant with modulus $1$.

\medskip

Suppose that $p\in \Omega_\lambda$. Let $U_{\lambda,p}$ be an open ball with center at $p$ that lies in a single coordinate chart with respect to
the standard coordinate system of $\mathbb{P}^k$. Then $\pi^{-1}(U_{\lambda,p})$ can be identified with $\mathbb{C}^{\ast} \times U_{\lambda,p}$ in canonical way and each point of
$\pi^{-1}(U_{\lambda,p})$ can be written as $(c,z)$. On $\pi^{-1}(U_{\lambda,p})$, the function $G_{\la}$ has the form
\begin{equation}
G_\la(c,z)=\log|c|+\gamma_\la(z)
\end{equation}
by (3.2). Assume that there is a section $s_\lambda$ such that
$s_\la(U_{\la,p})\subset \partial \mathcal{A}_\lambda$. Note that $s_\lambda(z)=(\sigma_\la(z),z)$ in $U_{\la,p}$ where $\sigma_\la$ is a non--vanishing holomorphic function on
$U_{\la,p}$. By Proposition 1.6, $G_\lambda\circ s_\la=0$ on $U_{\la,p}$. Thus
\[
0=G_\lambda\circ s_\lambda(z)=\log|\sigma_\lambda(z)|+\gamma_\lambda(z).
\]
Thus $\gamma_\lambda(z)=-\log \vert \sigma_\lambda(z)\vert$ is pluriharmonic on $U_{\lambda,p}$ and consequently $G_\lambda$ is pluriharmonic on
$\pi^{-1}(U_{\lambda,p})$ by (3.4). On the other hand
suppose that $\gamma_\la$ is pluriharmonic. Then there exists a conjugate function $\gamma_\la^{\ast}$ on $U_{\la,p}$
such that $\gamma_\la+i\gamma_\la^{\ast}$ is holomorphic. Define $\sigma_\la(z)=\exp (-\gamma_\la(z)-i\gamma_\la^{\ast}(z))$ and $s_\la(z)=(\sigma_\la(z),z)$.
Then $G_\la(s_\la(z))=\log |\sigma_\la(z)|+\gamma_\la(z)=0$ which shows that $s_\la(U_{\la,p})\subset \partial \mathcal{A}_\la$.

\medskip

\no {\it Step 2:} $\Om_{\la} = \Om'_{\la}$ for each $\la \in M$.

\medskip

Let $p\in \Omega_\la'$ and suppose that $U_{\la,p}$ is a neighborhood of $p$ on which there is a subsequence of
\[
\{f_{\sigma^{j-1}(\la)}\circ \ldots \circ f_\la\}_{j\geq 1}
\]
which is  uniformly convergent. Without loss of generality we may assume that
\[
g_\la = \lim_{j\ra\infty} f_{\sigma^{j-1}(\la)}\circ \ldots \circ f_\la
\]
on $U_{\la, p}$. By rotating the homogeneous coordinates $[x_0:x_1: \ldots : x_k]$ on $\mathbb{P}^k$, we may assume that
$g_\la(p)$ avoids the hyperplane at infinity $H = \big\{x_0=0\big\}$ and that $g_\lambda(p)$ is of the form $[1:g_1: \ldots : g_k]$.
Now choose an $\epsilon$ neighborhood
\[
N_\epsilon=\big\{\vert x_0 \vert < \epsilon {\big({\vert x_0 \vert}^2+ \ldots +{\vert x_k \vert}^2\big)}^{1/2} \big\}
\]
of $\pi^{-1}(H)$ in $\mathbb{C}^{k+1}\setminus \big\{0\big\}$ so that
\[
1>\epsilon {\big(1+{\vert g_1 \vert}^2+ \ldots +{ \vert g_k \vert}^2\big)}^{1/2}.
\]
Clearly $g_\lambda(p)\notin \pi(N_\epsilon)$. Shrink $U_{\la,p}$ if needed so that
\[
f_{\sigma^{j - 1}(\la)}\circ \ldots \circ f_\la (U_{\la,p})
\]
is uniformly separated from $\pi(N_\epsilon)$ for sufficiently large $l$. Define
\[
s_\la(z)=
\begin{cases}
\log  \Vert z \Vert & ;\text{ if } z\in N_\epsilon, \\
\log(\vert z_0 \vert / \vert \epsilon \vert ) & ;\text{ if } z\in \mathbb{C}^{k+1}\setminus (N_\epsilon \cup \{0\})
\end{cases}
\]

\no Note that $0\leq s(z)-\log \Vert z \Vert \leq \log(1/\epsilon)$ which implies that
\[
d^{-{j}}s_\lambda (f_{\sigma^{j-1}(\la)}\circ \ldots \circ f_\la(z))
\]
converges uniformly to the Green function $G_\la$ as $j\ra \infty$ on $\mathbb{C}^{k+1}$. Further if $z\in \pi^{-1}(U_{\la,p})$, then
\[
F_{\sigma^{j-1}(\la)}\circ \ldots \circ F_\la (z)\in \mathbb{C}^{k+1}\setminus (N_\epsilon \cup \{0\}).
\]
This shows that $d^{-{j}}s_\lambda(f_{\sigma^{j-1}(\la)}\circ \ldots \circ f_\la(z))$ is pluriharmonic in  $\pi^{-1}(U_{\la,p})$ and as a consequence the limit function $G_\la$ is also
pluriharmonic in $ \pi^{-1}(U_{\la,p})$. Thus $p\in \Omega_\la$.

\medskip

Now pick a point $p\in \Omega_\lambda$. Choose a neighborhood $U_{\lambda,p}$ of $p$ and a section $s_\lambda: U_{\lambda,p}\ra \mathbb{C}^{k+1}$ as in Step 1.
Since $F_{\la} : \cal A_{\la} \ra \cal A_{\si(\la)}$ is a proper map for each $\la$, it follows that
\[
(F_{\sigma^{j-1}(\lambda)}\circ \ldots \circ F_{\sigma(\lambda)}\circ
F_\lambda)(s_\lambda(U_{\lambda, p}))\subset \partial \mathcal{A}_{\sigma^j(\lambda)}.
\]
It was noted earlier that there exists a $R > 0$ such that $\Vert F_{\la}(x) \Vert \ge 2 \Vert x \Vert$ for all $\la$ and $\Vert x \Vert \ge R$. This shows that
$\mathcal{A}_\lambda\subset {B}_R$ for all $\lambda\in M$, which in turn implies that the sequence
\[
\big\{(F_{\sigma^{j-1}(\lambda)}\circ \ldots \circ F_{\sigma(\lambda)}\circ F_\lambda)\circ s_\lambda\big\}_{j\geq 0}
\]
is uniformly bounded on $U_{\la, p}$. We may assume that it converges and let $g_\lambda:U_{\lambda,p} \ra \mathbb{C}^{k+1}$ be its limit function.
Then $g_\lambda(U_{\lambda,p})\subset \mathbb{C}^{k+1}\setminus \{0\}$ since all the boundaries $\pa \cal A_{\la}$ are at a uniform distance away from the origin; indeed,
recall that there exists a uniform $r > 0$ such that the ball ${B}_r \subset \mathcal{A}_\lambda$ for all $\lambda\in M$. Thus $\pi \circ g_\lambda$ is
well defined and the sequence
$\big\{f_{\sigma^{j-1}(\lambda)}\circ \ldots \circ f_{\sigma(\lambda)}\circ f_\lambda\big\}_{j\geq 0}$ converges to $\pi\circ g_\lambda$ uniformly on compact sets.
Thus $\big\{f_{\sigma^{j-1}(\lambda)}\circ \ldots \circ f_{\sigma(\lambda)}\circ f_\lambda \big\}_{j\geq 0}$ is a normal family in $U_{\lambda,p}$. Hence $p\in \Omega_{\lambda}'$.

\medskip

\no {\it Step 3:} Each $\Om_{\la}$ is pseudoconvex and Kobayashi hyperbolic.

\medskip

That $\Om_{\la}$ is pseudoconvex follows exactly as in Lemma 2.4 of \cite{U}. To show that $\Omega_\lambda$ is Kobayashi hyperbolic, it suffices to prove that each component $U$ of
$\Omega_\lambda$ is Kobayashi hyperbolic. For a point $p$ in $U$ choose $U_{\lambda,p}$ and $s_\lambda$ as in Step $1$. Then $s_\lambda$ can be analytically continued to $U$. This
analytic continuation of $s_\lambda$ gives a holomorphic map $\tilde{s}_{\lambda}: \widetilde{U}\ra \mathbb{C}^{k+1}$ satisfying $\pi\circ \tilde{s}_{\lambda}=p$ where $\widetilde{U}$
is
a covering of $U$ and $p: \widetilde{U}\ra U$ is the corresponding covering map. Note that there exists a uniform $R>0$ such that $\lVert F_\lambda(z)\rVert \geq 2 \lVert z \rVert$ for
all $\lambda\in M$ and for all $z\in \mathbb{C}^{k+1}$  with $\lVert z \rVert \geq R$. Thus $\mathcal{A}_\lambda \subset B(0,R)$ and  $\tilde{s}_{\lambda}(\widetilde{U})\subset
B(0,2R)$.
Since $\tilde{s}_{\lambda}$ is injective and $B(0,2R)$ is Kobayashi hyperbolic in $\mathbb{C}^{k+1}$, it follows that $\widetilde{U}$ is Kobayashi hyperbolic. Hence $U$ is Kobayashi
hyperbolic.



\begin{thebibliography}{MNTU}

\bibitem{BS1} E. Bedford, J. Smillie:
\emph{Polynomial diffeomorphisms of $\mbb C^2$: Currents, equilibrium measure and hyperbolicity}, Invent. Math. {\bf 103} (1991), pp. 69--99.

\bibitem{BS2} E. Bedford, J. Smillie:
\emph{Polynomial diffeomorphisms of $\mbb C^2$ -- II: Stable manifolds and recurrence}, J. Amer. Math. Soc. {\bf 4} (1991), pp. 657--679.

\bibitem{BS3} E. Bedford, J. Smillie:
\emph{Polynomial diffeomorphisms of $\mbb C^2$ -- III: Ergodicity, exponents and entropy of the equilibrium measure}, Math. Ann. {\bf 294} (1992), pp. 395--420.

\bibitem{CF} D. Coman, J. E. Fornaess:
\emph{Green's functions for irregular quadratic polynomial automorphisms of $\mbb C^3$}, Michigan Math. J. {\bf 46} (1999), no. 3, pp. 419--459

\bibitem{CG} D. Coman, V. Guedj:
\emph{Invariant currents and dynamical Lelong numbers}, J. Geom. Anal. {\bf 14} (2004), no. 2, pp. 199--213

\bibitem{T1} H. De Th\'{e}lin:
\emph{Endomorphismes pseudo-altoires dans les espaces projectifs I}, Manuscripta Math. {\bf 142} (2013), no. 3-4, pp. 347--367.

\bibitem{T2} H. De Th\'{e}lin:
\emph{Endomorphismes pseudo-altoires dans les espaces projectifs II}, J. Geom. Anal. {\bf 25} (2015), no. 1, 204--225.

\bibitem{DS} T. C. Dinh, N. Sibony:
\emph{Rigidity of Julia sets for H\'{e}non type maps}, Proceedings of the 2008-2011 Summer Institute at Bedlewo. Modern Dynamics and its Interaction with Analysis, Geometry and Number
Theory, To appear.

\bibitem{FJ} C. Favre, M. Jonsson:
\emph{Dynamical compactifications of $\mbb C^2$}, Ann. of Math. (2) {\bf 173} (2011), no. 1, pp. 211--248.

\bibitem{FS} J. E. Fornaess, N. Sibony:
\emph{Complex H\'{e}non mappings in $\mbb C^2$ and Fatou--Bieberbach domains}, Duke Math. J. {\bf 65} (1992), pp. 345--380.

\bibitem{FW} J. E. Fornaess, H. Wu:
\emph{Classification of degree $2$ polynomial automorphisms of $\mbb C^3$}, Publ. Mat. {\bf 42} (1998), pp. 195--210.

\bibitem{G} V. Guedj:
\emph{Courants extraux et dynamique complexe}, Ann. Sci. \'{E}cole Norm. Sup. {\bf 4} 38 (2005), no. 3, pp. 407--426.

\bibitem{GS} V. Guedj, N. Sibony:
\emph{Dynamics of polynomial automorphisms of $\mbb C^k$}, Ark. Mat. {\bf 40} (2002), pp. 207--243.

\bibitem{HP} J. H. Hubbard, P. Papadopol:
\emph{Superattractive fixed points in $\mbb C^n$}, Indiana Univ. Math. J. {\bf 43} (1994), no. 1, pp. 321--365.

\bibitem{JM} M. Jonsson:
\emph{Dynamics of polynomial skew products on $\mathbb{C}^2$}, Math. Ann. {\bf 314} (1999), 403--447.

\bibitem{J} M. Jonsson:
\emph{Ergodic properties of fibered rational maps}, Ark. Mat. {\bf 38} (2000), no. 2, pp. 281--317.

\bibitem{P} H. Peters:
\emph{Non-autonomous dynamics in $\mbb P^k$}, Ergodic Theory Dynam. Systems {\bf 25} (2005), no. 4, pp. 1295--1304.

\bibitem{PW} H. Peters, E. F. Wold:
\emph{Non-autonomous basins of attraction and their boundaries}, J. Geom. Anal. {\bf 15} (2005), no. 1, pp. 123--136.

\bibitem{S} N. Sibony:
\emph{Dynamique des applications rationnelles de $\mbb P^k$}, Panoramas et Synth\'{e}ses, {\bf 8} (1999), pp. 97--185.

\bibitem{U} T. Ueda:
\emph{Fatou sets in complex dynamics on projective spaces}, J. Math. Soc. Japan {\bf 46} (1994), no. 3, pp. 545--555

\bibitem{MNTU} S. Morosawa, Y. Nishimura, M. Taniguchi, T. Ueda:
\emph{Holomorphic dynamics}, Translated from the 1995 Japanese original and revised by the authors. Cambridge Studies in
Advanced Mathematics, {\bf 66}, Cambridge University Press, Cambridge, (2000).

\bibitem{Y}  Y. Yomdin: 
\emph {Volume growth and entropy}, Israel J. Math. {\bf 57} (1987), no. 3, 285–-300. 

\bibitem{W}  P. Walters:
\emph {An introduction to ergodic theory},  Graduate Texts in Mathematics, {\bf 79} Springer-Verlag, New York-Berlin, 1982. 


\end{thebibliography}
\end{document}